\documentclass[12pt]{article}
\usepackage[margin=1in]{geometry}
\usepackage{amsmath,amssymb,amsfonts,amsthm}
\usepackage{graphicx}
\usepackage{subcaption}

\theoremstyle{plain}
\newtheorem{theorem}{Theorem}[section]
\newtheorem{proposition}[theorem]{Proposition}
\newtheorem{conjecture}[theorem]{Conjecture}
\newtheorem{corollary}[theorem]{Corollary}
\newtheorem{property}[theorem]{Property}
\newtheorem{lemma}[theorem]{Lemma}

\theoremstyle{remark}

\newcommand{\Z}{\mathbb{Z}}
\DeclareMathOperator{\CR}{cr}
\newcommand{\Mod}[1]{\ (\mathrm{mod}\ #1)}

\title{On Kainen's conjectures on surface crossing numbers}
\author{Timothy Sun\\Department of Computer Science\\San Francisco State University}
\date{}

\begin{document}

\maketitle

\begin{abstract}
In 1972, Kainen proved a general lower bound on the crossing number of a graph in a closed surface and conjectured that this bound is tight when the graph is either a complete graph or a complete bipartite graph, and the surface is of genus close to the minimum genus of that graph. Prior to the present work, these conjectures were known to be true only for small cases and when the conjectures predict a crossing number of 0, i.e., when a triangular or quadrangular embedding was already known. We show that Kainen's conjectures are true except for the three graphs $K_9$, $K_{3,5}$, and $K_{5,5}$. We also prove nonorientable analogues of these conjectures, where the only exceptions to the general formulas are $K_7$ and $K_8$. 
\end{abstract}

\section{Introduction}

Perhaps the two best-known topological parameters for graphs are the \emph{(minimum, orientable) genus} and the \emph{(planar) crossing number}. Unless otherwise specified, we restrict ourselves to embeddings in orientable surfaces. For a graph $G$, let $\gamma(G)$ denote the genus of $G$. The genus formulas of the complete graphs $K_n$ and the complete bipartite graphs $K_{m,n}$ are:
\begin{equation}
\gamma(K_n) = H(n) := \left\lceil \frac{(n-3)(n-4)}{12}\right\rceil, \text{~for~} n \geq 3 \text{~and~}
\label{eq-genus-kn}
\end{equation}
\begin{equation}
\gamma(K_{m,n}) = H(m,n) := \left\lceil \frac{(m-2)(n-2)}{4}\right\rceil, \text{~for~} m,n \geq 2.
\label{eq-genus-kmn}
\end{equation}
Calculating the genus of the complete graphs is the bulk of the proof of the Map Color Theorem of Ringel, Youngs, and others \cite{Ringel-MapColor}. The genus of the complete bipartite graphs was first proven by Ringel \cite{Ringel-BipartiteOrientable}, though other proofs have been given, most notably by Bouchet \cite{Bouchet-Diamond}. These two formulas match the so-called ``Euler lower bound'' on the genus of simple graphs, which is derived from the Euler polyhedral equation and the girth of the graph. 

On the other hand, the crossing numbers of these graphs are unknown in general, and the best known drawings show that the crossing numbers are at most
\begin{equation*}
\CR(K_{n}) \leq \frac{1}{4} \bigg\lfloor \frac{n}{2}\bigg\rfloor \bigg\lfloor \frac{n-1}{2} \bigg\rfloor \bigg\lfloor \frac{n-2}{2}\bigg\rfloor \bigg\lfloor \frac{n-3}{2} \bigg\rfloor,\text{~and~}
\end{equation*}
\begin{equation*}
\CR(K_{m,n}) \leq \bigg\lfloor \frac{m-1}{2}\bigg\rfloor \bigg\lfloor \frac{m}{2} \bigg\rfloor \bigg\lfloor \frac{n-1}{2}\bigg\rfloor \bigg\lfloor \frac{n}{2} \bigg\rfloor.
\end{equation*}
These upper bounds are thought to be optimal (these conjectures are variously attributed to Harary, Hill, Guy, Zarankiewicz, etc.), but proving the matching lower bounds is a well-known and difficult open problem. Equality has been shown for some small cases \cite{PanRichter-K11, Kleitman-K5n, Woodall-Cyclic}.

A generalization of both of these parameters is the \emph{surface crossing number} $\CR_g(G)$, which is the minimum number of crossings needed for drawing $G$ in the orientable surface $S_g$, the sphere with $g$ handles. Kainen \cite{Kainen-CrossingLower} derived a lower bound for this parameter by noting that if one edge is deleted from each crossing, the drawing becomes an embedding, in which case the Euler polyhedral formula applies. Kainen's lower bound is far smaller than the conjectured values for the planar crossing numbers $\CR_0(K_n)$ and $\CR_0(K_{m,n})$, but when the genus of the surface is close to the minimum genus of the graph, the crossing number is often close to, or even exactly equal to, the lower bound. Kainen and White~\cite{KainenWhite-Stable} called this phenomenon ``stability."

Kainen \cite{Kainen-CrossingLower} conjectured a stability result for $K_n$ and $K_{m,n}$: if we replace the ceiling functions in the definition of $H(n)$ and $H(m,n)$ (see equations (\ref{eq-genus-kn}) and (\ref{eq-genus-kmn}) above) with floor functions, we obtain the expressions
\begin{align*}
h(n) &:= \left\lfloor \frac{(n-3)(n-4)}{12}\right\rfloor, \text{~and} \\
h(m,n) &:= \left\lfloor \frac{(m-2)(n-2)}{4}\right\rfloor. \\
\end{align*}
In the surfaces of orientable genus $h(n)$ and $h(m,n)$, Kainen conjectured that the crossing numbers of $K_n$ and $K_{m,n}$, respectively, equal the value obtained from his lower bound. Our proof of Kainen's conjectures can be stated in the following way:

\begin{theorem}
The surface crossing number of the complete graph $K_n$, $n \geq 3$, in the surface of genus $h(n)$ is $$\CR_{h(n)}(K_n) = \frac{(n-3)(n-4)}{2} \bmod{6},$$ except when $n = 9$. The crossing number of $K_9$ in the surface $S_2$ is 4. 
\label{thm-kn}
\end{theorem}

\begin{theorem}
The surface crossing number of the complete bipartite graph $K_{m,n}$, $m,n \geq 3$ in the surface of genus $h(m,n)$ is
$$\CR_{h(m,n)}(K_{m,n}) = (m-2)(n-2) \bmod{4},$$ except when $(m,n) = (3,5)$, $(5,3)$, or $(5,5)$. The crossing numbers of $K_{3,5}$ in $S_0$ and $K_{5,5}$ in $S_2$ are 4 and 2, respectively. 
\label{thm-kmn}
\end{theorem}

Kainen's lower bound and conjectures have direct analogues for nonorientable surfaces, but we defer the exact statements to Sections \ref{sec-kn-non} and \ref{sec-kmn-non}. In short, $K_7$ and $K_8$ are the only counterexamples to those formulas.

\section{Euler's formula and Kainen's lower bound}

For more background on topological graph theory, see Gross and Tucker \cite{GrossTucker} or Mohar and Thomassen \cite{MoharThomassen}. 

In this paper, we restrict our attention to embeddings and drawings of graphs in closed surfaces, where $S_g$ denotes the orientable surface of genus $g$, and $N_k$ denotes the nonorientable surface of genus $k$, i.e. the sphere with $k$ crosscaps. A \emph{drawing} of a graph $G$ in a surface $S$ is a locally injective map $D: G \to S$, where vertices do not intersect other vertices or the interiors of edges, and the number of \emph{crossings}, i.e., intersections of the interiors of two edges, is finite. The \emph{surface crossing number} $\CR_g(G)$ is the smallest number of crossings of any drawing of $G$ in the surface $S_g$. If a drawing has no crossings, then it is said to be an \emph{embedding}. An embedding $\phi: G \to S$ is \emph{cellular} if its \emph{faces}, i.e., the connected components of $S\setminus \phi(G)$, are homeomorphic to open disks. 

Given an arbitrary orientation of the edges $E(G)$, each edge $e \in E(G)$ induces two arcs $e^+$ and $e^-$ pointing in opposite directions. Let $E(G)^+$ denote the set of such arcs. A \emph{rotation} of a vertex is a cyclic permutation of the arcs leaving that vertex, and a \emph{rotation system} is an assignment of a rotation to each vertex. Up to orientation-preserving equivalence, rotation systems are in one-to-one correspondence with cellular embeddings in orientable surfaces: a cellular embedding induces a rotation system from the clockwise ordering of incident edges at each vertex, and the cellular embedding can be reconstructed from the rotation system via face-tracing.

Rotation systems can be generalized to possibly nonorientable surfaces by introducing an \emph{edge signature} $\lambda\colon E(G) \to \{-1, 1\}$. Edges with signature $-1$ are called \emph{twisted}, and traversing such an edge reverses the local orientation. The edges of signature $1$ are called \emph{normal}, and if every edge is normal, then we simply have an ordinary rotation system describing an embedding in an orientable surface. A rotation system describes a nonorientable embedding if there is a closed walk that traverses an odd number of edges of signature $-1$. 

Suppose we have a cellular embedding of a graph $G$ in a surface $S$. The Euler polyhedral formula states that
$$|V(G)|-|E(G)|+|F(G,\phi)| = \chi(S),$$
where $F(G,\phi)$ denotes the set of faces of the embedding, and $\chi(S)$ is the \emph{Euler characteristic} of $S$. For an orientable surface $S_g$, the Euler characteristic is $2-2g$, and for a nonorientable surface $N_k$, the Euler characteristic is $2-k$. The \emph{girth} of a graph is the length of its shortest cycle, and it can be applied to the Euler polyhedral formula to bound the number of edges in an embedded graph:

\begin{proposition}
Let $G$ be a simple, connected graph of girth $\ell < \infty$. If $G$ can be embedded in the surface $S$, then 
$$|E(G)| \leq \frac{\ell}{\ell-2}(|V(G)|-\chi(S)).$$
\end{proposition}
\begin{proof}
If the embedding is not cellular, then there exists an embedding in a surface of larger Euler characteristic, leading to an even stronger inequality. Since $G$ has a cycle, each face must contain a cycle and hence has length at least the girth $\ell$. Thus, $2|E| \geq \ell|F|$. Substituting this inequality into the Euler polyhedral equation yields:
$$\chi = |V|-|E|+|F| \leq |V|-|E|+\frac{2}{\ell}|E| = |V|+\frac{2-\ell}{\ell}|E|,$$
which can be rearranged into the desired upper bound on $|E|$. 
\end{proof}

How much the graph's actual edge count exceeds this value is a lower bound on the number of crossings of any drawing of $G$:

\begin{proposition}[Kainen's lower bound \cite{Kainen-CrossingLower}]
For a simple, connected graph $G$ of girth $\ell < \infty$, the crossing number of $G$ in the surface $S_g$ is at least
$$\CR_g(G) \geq \delta_g(G) := |E(G)| - \frac{\ell}{\ell-2}(|V(G)|-2+2g).$$
\end{proposition}
\begin{proof}
In a drawing of $G$ in $S_g$ with $\CR_g(G)$ crossings, deleting one edge from each crossing yields a simple graph embedded in $S_g$. This graph cannot be disconnected or acyclic, otherwise one can add back at least one of those deleted edges without creating any crossings. 
\end{proof}

In this work, we compute surface crossing numbers of $K_n$, the \emph{complete graph} on $n$ vertices, and $K_{m,n}$, the \emph{complete bipartite graph} on $m$ \emph{left} vertices and $n$ \emph{right} vertices. In our figures, left and right vertices are colored black and white, respectively.

We define a \emph{Kainen drawing} to be a drawing of a graph $G$ where the number of crossings is exactly equal to Kainen's lower bound, showing that $\CR_g(G) = \delta_g(G)$. A \emph{Kainen subembedding} of $G$ is an embedding of a spanning subgraph $G'$ where each of the missing edges $E(G)\setminus E(G')$ can be sequentially added to the drawing using one additional crossing each to obtain a Kainen drawing. Like in the proof of Kainen's lower bound, a Kainen subembedding can be obtained from a Kainen drawing by deleting one edge from each crossing. Unless we state otherwise, we consider only Kainen drawings and subembeddings of the graphs $K_n$ and $K_{m,n}$ in the surfaces $S_{h(n)}$ and $S_{h(m,n)}$, respectively. 

It is often helpful to examine how many crossings an edge or neighborhood of a vertex participates in. The \emph{responsibility} of an edge is the number of crossings that it participates in, and the \emph{responsibility} of a vertex is the sum of the responsibilities of its incident edges. If a vertex has responsibility 0, we call that vertex a \emph{slacker}. In a Kainen drawing, edges cannot have responsibility more than 1, otherwise we would be able to delete it to obtain a drawing violating Kainen's lower bound. 

Kainen \cite{Kainen-CrossingLower} showed that the highest-genus surface where his lower bound is nonnegative for the graphs $K_n$ and $K_{m,n}$ is exactly $h(n)$ and $h(m,n)$, respectively. More explicitly, these lower bounds are:

\begin{proposition}
$$cr_{h(n)}(K_{n}) \geq \frac{(n-3)(n-4)}{2} \bmod{6}$$
and
$$cr_{h(m,n)}(K_{m,n}) \geq (m-2)(n-2) \bmod{4}.$$
\end{proposition}
\begin{proof}
Let $t(n)$ denote Kainen's lower bound for $K_n$ in $S_{h(n)}$. By direct substitution, we have
\begin{align*}
t(n) &= \frac{n(n-1)}{2} - \left(3n-6+6\left\lfloor \frac{(n-3)(n-4)}{12}\right\rfloor\right) \\
     &= (n-3)(n-4)/2-6\left\lfloor \frac{(n-3)(n-4)/2}{6}\right\rfloor \\
     &< 6,
\end{align*}
where the inequality comes from the fact that $\lfloor x \rfloor > x-1$. Thus, $t(n) = t(n) \bmod 6 = (n-3)(n-4)/2 \bmod 6$. The complete bipartite case is calculated similarly. 
\end{proof}

Like in the proof of the Map Color Theorem, we can separate the complete graphs according to their residue modulo 12, where within each residue, Kainen's lower bound yields the same number of crossings. Using the notation above, we observe that
$$t(n) = \begin{cases}
0 & \text{if~} n \equiv 0, 3, 4, 7 \Mod{12} \\
1 & \text{if~} n \equiv 2, 5 \Mod{12} \\
3 & \text{if~} n \equiv 1, 6, 9, 10 \Mod{12} \\
4 & \text{if~} n \equiv 8, 11 \Mod{12}. \\
\end{cases}$$
Let $t(m,n)$ denote Kainen's lower bound for $K_{m,n}$ in $S_{h(m,n)}$. Then 
$$t(m,n) = \begin{cases}
0 & \text{if~} m \equiv 2 \Mod{4}, \text{~or~} n \equiv 2 \Mod{4}, \text{~or~} m,n \equiv 0 \Mod{2} \\
1 & \text{if~} m,n \equiv 1 \Mod{4} \text{~or~} m,n \equiv 3 \Mod{4} \\
2 & \text{if~} m \equiv 0 \Mod{4}, n \equiv 1 \Mod{2} \text{~or~} m \equiv 1 \Mod{2}, n \equiv 0 \Mod{4} \\
3 & \text{if~} m \equiv 1 \Mod{4}, n \equiv 3 \Mod{4} \text{~or~} m \equiv 3 \Mod{4}, n \equiv 1 \Mod{4}. \\
\end{cases}$$

Our main results, Theorems \ref{thm-kn} and \ref{thm-kmn}, show that these are equalities except for three exceptional graphs. Kainen subembeddings are found using the same tools used in the construction of genus embeddings of those same families of graphs: current graph constructions for $K_n$, and the diamond sum operation for $K_{m,n}$. 

We briefly discuss the exceptions to the general formula. Huneke \cite{Huneke-Minimum} showed that no simple graph on 9 vertices has a triangular embedding in $S_2$, ruling out the existence of any Kainen subembedding of $K_9$. Riskin \cite{Riskin-Genus2} gave an independent proof of the weaker fact that $\CR_2(K_9) > 3$ and found a drawing which uses 4 crossings. There are two complete bipartite counterexamples. It was already known (see, e.g., Kleitman \cite{Kleitman-K5n}) that the planar crossing number of $K_{3,5}$ is $4$. The remaining counterexample, $K_{5,5}$ in $S_2$, will be treated in Section \ref{sec-k55}. Mohar \emph{et al.} \cite{MPP-NearlyComplete} found a quadrilateral embedding of $K_{5,5}-K_2$ in $S_2$, but the missing edge cannot be drawn into their embedding using only one crossing. 

\section{Related work}

In the cases where $H(n) = h(n)$ and $H(m,n) = h(m,n)$, the predicted drawing is actually an embedding where the length of every face is equal to the girth of the graph. When $n \geq 3$, the complete graph $K_n$ has an orientable triangular embedding if and only if $n \equiv 0, 3, 4, 7 \Mod{12}$ (see, e.g., Ringel \cite{Ringel-MapColor}). Ringel \cite{Ringel-BipartiteOrientable} showed that the complete bipartite graph $K_{m,n}$ has a quadrangular embedding if and only if $(m-2)(n-2) \equiv 0 \Mod{4}$. Outside of these residues, we use only a few previously known constructions. One exception is when $n \equiv 5 \Mod{12}$: the triangular embeddings of $K_n - K_2$ presented in Youngs \cite{Youngs-3569} are also Kainen subembeddings of $K_n$. 

Kainen subembeddings of complete graphs are examples of \emph{minimal triangulations}, triangular embeddings of simple graphs on some surface with the fewest number of vertices possible. Minimal triangulations were studied by Jungerman and Ringel \cite{JungermanRingel-Minimal}, who gave an example of a minimal triangulation for each orientable surface. A Kainen drawing with more than one crossing can be seen as describing multiple minimal triangulations simultaneously, since there is a choice of which edge to delete from each crossing. In this sense, Kainen drawings with multiple crossings are usually much harder to find than arbitrary minimal triangulations, though to answer Kainen's conjecture for complete graphs, we only need to find the requisite drawings in a small subset of all orientable surfaces. 

A complete bipartite analogue of the minimal triangulations problem was considered by Magajna \emph{et al.} \cite{Magajna-MinimalOrdered}, who were the first to interpret Bouchet's \cite{Bouchet-Diamond} diamond sum operation in primal form. Their result was later strengthened to the genus of ``nearly complete'' bipartite graphs \cite{MPP-NearlyComplete, Mohar-Nonorientable, Lv-CompleteBipartite, Singh-Settling}, which are complete bipartite graphs with a matching deleted. 

Guy, Jenkyns, and Schaer \cite{Guy-ToroidalComplete, Guy-ToroidalCompleteBipartite} proved a variety of results concerning the surface crossing number of $K_n$ and $K_{m,n}$ in the torus. One such result is that the crossing number of $K_{3,n}$ in the torus is exactly $\lfloor (n-3)^2/12 \rfloor$. Their result was generalized by Richter and \v{S}ir\'a\v{n}~\cite{RichterSiran-K3n}, who calculated the surface crossing number of $K_{3,n}$ in every closed surface.

\section{Kainen drawings of $K_n$}

\subsection{Current graphs}

A \emph{current graph} consists of a graph $G$ equipped with an embedding $\phi\colon G \to S$ and an arc-labeling $\alpha\colon E(G)^+ \to \Gamma$ that satisfies $\alpha(e^+) = -\lambda(e)\alpha(e^-)$ for every edge $e \in E(G)$, where $\lambda$ is the edge signature of $\phi$. In this section, we only consider current graphs embedded in orientable surfaces, where the edge signature is $1$ for every edge. The group $\Gamma$ is called the \emph{current group}, and elements from $\Gamma$ are called \emph{currents}. In all of the current graphs described in the present work, $\Gamma$ is always a finite cyclic group $\mathbb{Z}_n$. We call each face boundary walk, which consists of a cyclic sequence of arcs $(e_1^\pm, e_2^\pm, \dotsc)$, a \emph{circuit}, and the \emph{log} of a circuit $(\alpha(e_1^\pm), \alpha(e_2^\pm), \dotsc)$ replaces each arc with its label. The \emph{index} of a current graph is the number of faces in the embedding, and it is required to divide the order of the current group. Our index $k$ current graphs satisfy the following properties: 

\begin{enumerate}
\item[(C1)] The circuits are labeled $[0], [1], \dotsc, [k-1]$.
\item[(C2)] The log of each circuit contains each nonzero element of $\Z_n$ exactly once.
\item[(C3)] If circuits $[a]$ and $[b]$ traverse the arcs $e^+$ and $e^-$, respectively, then $\alpha(e^+) \equiv b-a \Mod{k}$.   
\end{enumerate}

The \emph{derived embedding} of a current graph satisfying these properties is of a complete graph $K_n$, where the rotation at vertex $i \in \Z_n$ is found by taking the log of circuit $[i \bmod{k}]$ and adding $i$ to each element. Each vertex in the current graph induces a set of faces in the derived embedding. The number of such faces and their lengths can be calculated from the vertex's \emph{excess}, the sum of the incoming currents. The term ``current graph'' comes from the fact that most vertices satisfy \emph{Kirchhoff's current law} (KCL), i.e., have excess 0. Most of the families of current graphs in this paper contain three standard types of vertices:

\begin{enumerate}
\item[(V1)] Unlabeled vertex of degree 3, which has excess 0 (i.e., satisfies KCL).
\item[(V2)] Unlabeled vertex of degree 1, which has excess of order 2 in $\Gamma$.
\item[(V3)] Labeled vertex of degree $k$ (i.e., equal to the index of the graph), which is incident with each circuit, and whose excess generates the index $k$ subgroup of $\Gamma$. 
\end{enumerate}

Later on, we describe a few other types of vertices for use in specific constructions. One can show that a (V1) vertex generates triangular faces, a (V2) vertex generates 2-sided faces (which can be suppressed by identifying the two parallel edges), and a (V3) vertex generates one Hamiltonian face. (V3) vertices are called \emph{vortices} and are labeled by a letter, and we subdivide each resulting Hamiltonian face with a new vertex with the same letter. We call these subdivision vertices \emph{lettered vertices}, while the ones from the original derived embedding are called \emph{numbered vertices}. 

\subsection{General approach}\label{sec-general}

Many of the forthcoming constructions are extensions of previously known ideas for finding minimal triangulations \cite{JungermanRingel-Minimal, Sun-Minimum}. To illustrate some of the additional difficulties in finding a Kainen subembedding, we present a result due to Jonathan\ L.\ Gross (personal communication). Perhaps one of the simplest cases in the genus of $K_n$ is ``Case 10,'' i.e., when $n \equiv 10 \Mod{12}$. Finding genus embeddings of $K_{12s+10}$ starts with a simple family of index 1 current graphs. The smallest such current graph, shown in Figure \ref{fig-k10current}(a), generates a triangular embedding of $K_{10}-K_3$ in $S_3$. However, it is not a Kainen subembedding because none of the missing edges $(x,y)$, $(y, z)$, or $(x, z)$ can be added to the embedding using only one crossing. One can infer this from the current graph alone: the three vortices are at distance 2 from one another. The modification shown in Figure \ref{fig-k10current}(b) replaces the edges $(0,2)$, $(2, 6)$, and $(6,0)$ with $(x, y)$, $(y, z)$, $(z, x)$. The three edges that are now missing can be introduced elsewhere in the embedding with one crossing each. 

\begin{figure}[tbp]
\centering
    \begin{subfigure}[b]{0.38\textwidth}
        \centering
        \includegraphics[scale=1]{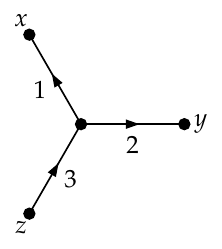}
        \caption{}
    \end{subfigure}
    \begin{subfigure}[b]{0.61\textwidth}
        \centering
        \includegraphics[scale=0.9]{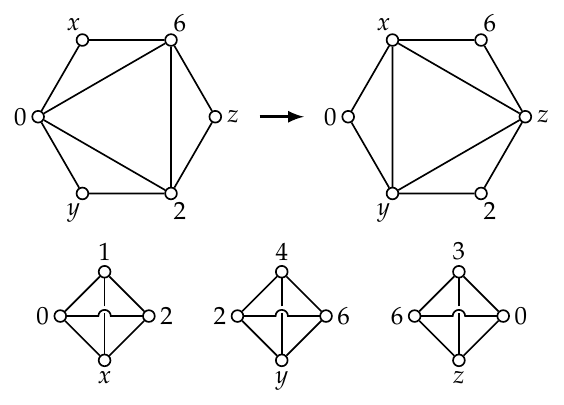}
        \caption{}
    \end{subfigure}
\caption{An index 1 current graph with group $\mathbb{Z}_{7}$ (a), and modifications to its derived embedding.}
\label{fig-k10current}
\end{figure}

Gross further observed that this approach does not seem to generalize, since vortices in larger members of this family of current graphs are even farther apart. What was an easy case in the original Map Color Theorem and minimum triangulations problem turns out to be the most difficult residue to solve in Kainen's conjecture. Nonetheless, this \emph{ad hoc} example illustrates a general theme in finding Kainen subembeddings. Initially, the missing edges of a derived embedding are between lettered vertices. In general, such an embedding will not be a Kainen subembedding, but a promising approach would be to modify the embedding so that some of the missing edges are between numbered vertices, instead. 

The derived embeddings of index 1 current graphs are often not conducive to these modifications, since vortices of type (V3) cannot be adjacent (except in trivial cases). However, it might be useful to ``lift'' to a higher index: let $k$, $\ell$, and $n$ be distinct positive integers such that $k$ divides $\ell$ and $\ell$ divides $n$. Given an index $k$ current graph with current group $\mathbb{Z}_n$, it can be interpreted as a symmetric index $\ell$ current graph. This suggests the possibility of other current graphs of the same index $\ell$ with more desirable properties, such as having adjacent vortices. Here, Case 10 hits another unfortunate roadblock: the current graphs have current group $\mathbb{Z}_{12s+7}$, whose order is prime infinitely often by Dirichlet's theorem. Our solution for this residue will ultimately follow a completely different approach.

\subsection{A note about notation}

In order to describe an infinite family of current graphs, we use standard repeating structures known as ``ladders.'' In every ladder in this paper, the ``rungs,'' i.e., the vertical edges, always alternate in direction, have currents that form an arithmetic sequence of step size $3$ or $-3$ when read from left to right, and only contain vertices that satisfy Kirchhoff's current law. For index 1 and 2 current graphs, the only other edges are drawn horizontally, and their currents can be calculated using KCL. The rotations on the vertices simply repeat for index 1 current graphs, but for index 2 current graphs, they form a checkerboard pattern. For index 3 current graphs, the rungs alternate between a ``simple'' rung and a ``ring-shaped'' rung. The currents on the curved edges of the latter match that of the horizontal edges above and below it. Figures \ref{fig-index2ladder} and \ref{fig-index3ladder} show two examples of infinite families of ladders. 

\begin{figure}[tbp]
\centering
\includegraphics[scale=0.9]{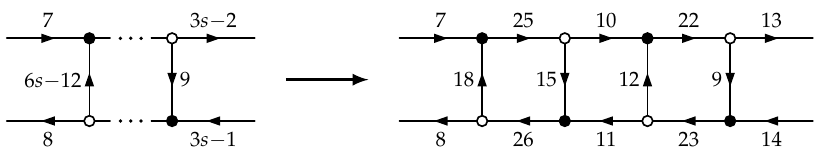}
\caption{An index 2 ladder and its specification for $s = 5$.}
\label{fig-index2ladder}
\end{figure}

\begin{figure}[tbp]
\centering
\includegraphics[width=\textwidth]{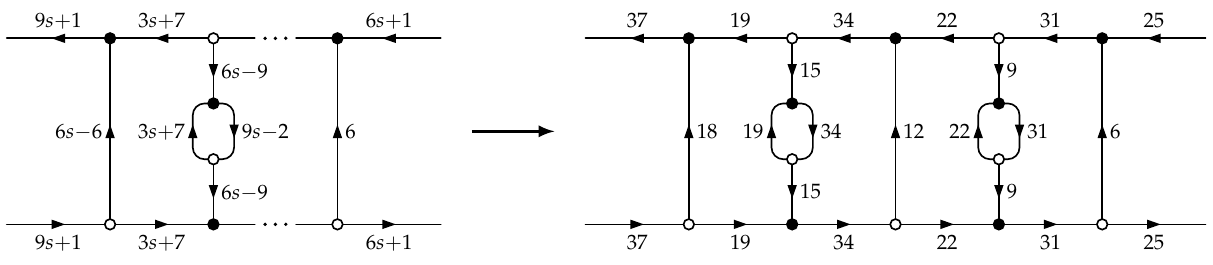}
\caption{An index 3 ladder and its specification for $s = 4$.}
\label{fig-index3ladder}
\end{figure}

Many of our constructions use sequences of edge flips to rearrange the locations of edges. If we write a sequence of edges like $$(u_1, v_1) \to (u_2, v_2) \to (u_3, v_3) \to \dotsc \to (u_i, v_i),$$ then we perform the following operations as shown in Figure \ref{fig-generalflips}: delete $(u_1, v_1)$, $(u_2, v_2)$, $\dotsc$, $(u_{i-1}, v_{i-1})$ and add $(u_2, v_2)$, $(u_3, v_3)$, $\dotsc$, $(u_i, v_i)$, respectively, in the resulting quadrilateral faces. 

\begin{figure}[tbp]
\centering
\includegraphics[width=\textwidth]{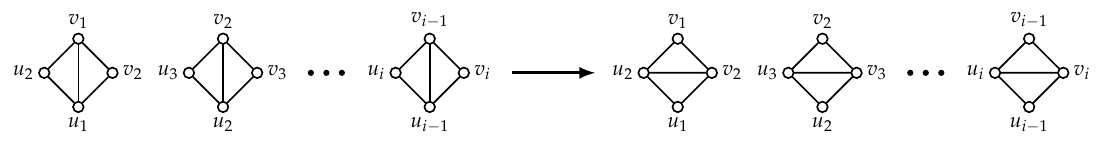}
\caption{A sequence of edge flips.}
\label{fig-generalflips}
\end{figure}

\subsection{The constructions}

As mentioned earlier, we solve each residue $n \bmod{12}$ individually. The residues are ordered by increasing number of crossings, and then in increasing order of complexity. 

When the complete graph $K_n$ has a triangular embedding, that embedding is automatically a Kainen drawing with 0 crossings:

\begin{lemma}
When $n \equiv 0, 3, 4, 7 \Mod{12}$, $n \geq 3$, the crossing number of $K_n$ in the surface of genus $h(n)$ is 0.
\label{lem-case0347}
\end{lemma}

For the simplest known constructions, see Sections 2.3, 6.1, and 9.2 of Ringel \cite{Ringel-MapColor} and Sun \cite{Sun-K12s}. 

\begin{lemma}
When $n = 12s+5$, $s \geq 0$, the surface crossing number of $K_n$ in the surface of genus $h(n)$ is $1$.
\label{lem-case5}
\end{lemma}
\begin{proof}
The triangular embeddings of $K_{12s+5}-K_2$ constructed in Youngs \cite{Youngs-3569} (see also Section 9.2 of Ringel \cite{Ringel-MapColor}) are Kainen subembeddings of $K_{12s+5}$, since the two vortices are adjacent.
\end{proof}

Triangular embeddings of $K_{12s+2}-K_2$ have been found using index 2 current graphs \cite{RingelYoungs-German, Jungerman-KnK2, Sun-Index2}. As a consequence of a ``global'' Kirchhoff's current law, the vortices cannot be adjacent in such a current graph. Lifting to index 4 is a potential workaround, but as seen in Pengelley and Jungerman \cite{Pengelley-Index4} and Korzhik \cite{Korzhik-Index4}, it is difficult to find infinite families of current graphs of such a high index. Our general solution consists of index 2 current graphs, with additional edge flips that bring the lettered vertices closer together in the rotation of some vertex. 

\begin{lemma}
When $n = 12s+2$, $s \geq 1$, the surface crossing number of $K_n$ in the surface of genus $h(n)$ is $1$.
\label{lem-case2}
\end{lemma}
\begin{proof}
For $s = 1$, the index 4 current graph in Figure \ref{fig-case2-s1} generates a triangular embedding of $K_{14}-K_2$. Since the two vortices are adjacent, this is a Kainen subembedding of $K_{14}$. For $s \geq 2$, consider the current graphs in Figures \ref{fig-current-c2-even} and \ref{fig-current-c2-odd}. In the derived embeddings, the rotation at vertex $0$ is of the form
$$\begin{array}{rlllllllllllllllllllllllllllll}
0. & \dots & x & u & 12s-3 & y \dots,
\end{array}$$
where $u = 6s-1$ or $1$ if $s$ is even or odd, respectively. After applying the sequence of edge flips 
$$(12s-3, x) \to (6s-4, 6s-2) \to (0, u) \to (12s-3, x)$$
for even $s$ and 
$$(12s-3, x) \to (12s-6, 12s-4) \to (0, u) \to (12s-3, x)$$ 
for odd $s$, the edge $(x,y)$ can be drawn with 1 crossing, as seen in Figure \ref{fig-case2-flips}.
\end{proof}

\begin{figure}[tbp]
\centering
\includegraphics[scale=0.9]{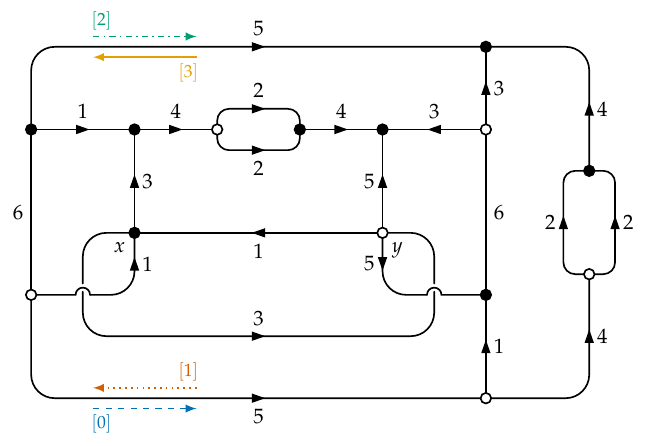}
\caption{Index 4 current graph with group $\mathbb{Z}_{12}$.}
\label{fig-case2-s1}
\end{figure}

\begin{figure}[tbp]
\centering
\includegraphics[scale=0.9]{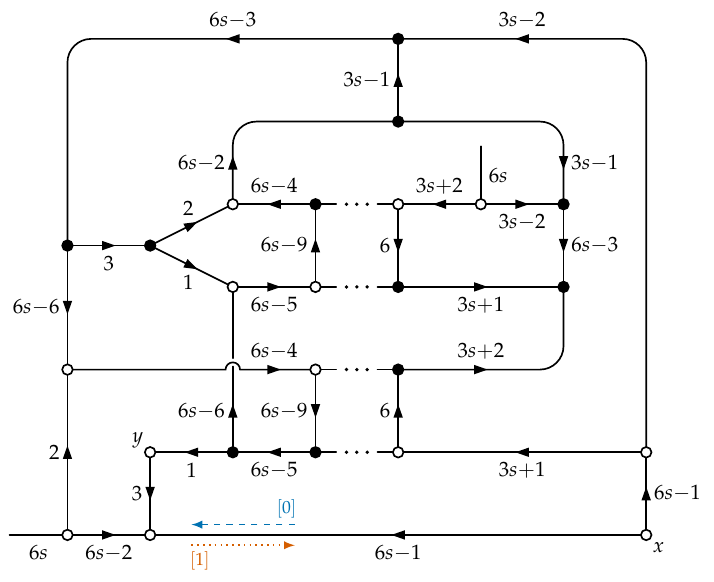}
\caption{Index 2 current graphs with current group $\mathbb{Z}_{12s}$, even $s \geq 2$.}
\label{fig-current-c2-even}
\end{figure}

\begin{figure}[tbp]
\centering
\includegraphics[scale=0.9]{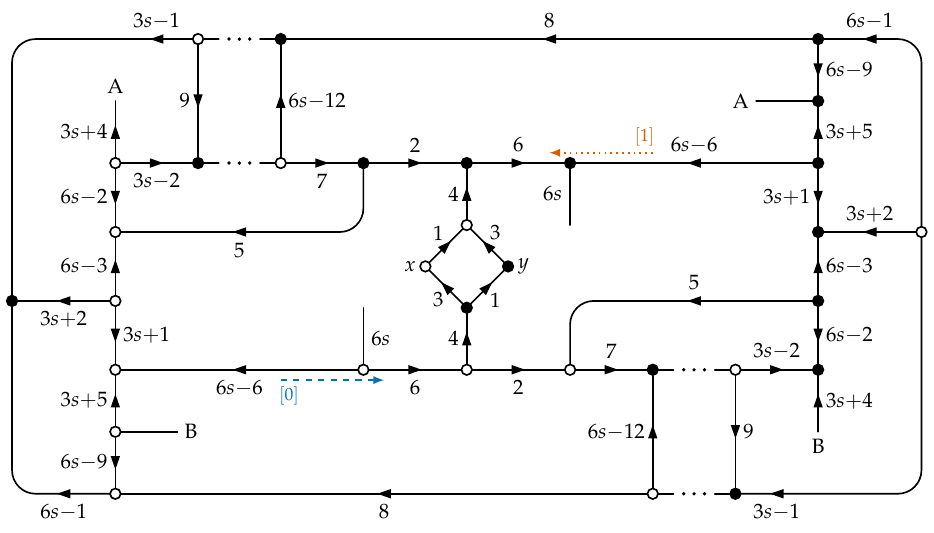}
\caption{Index 2 current graphs with current group $\mathbb{Z}_{12s}$, odd $s \geq 3$.}
\label{fig-current-c2-odd}
\end{figure}

\begin{figure}[tbp]
\centering
\includegraphics[scale=0.9]{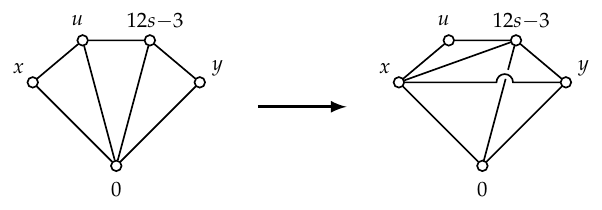}
\caption{Modifications for drawing $K_{12s+2}$, $s \geq 2$, with one crossing.}
\label{fig-case2-flips}
\end{figure}

For $n \equiv 6, 9 \Mod{12}$, we find two new families of index 3 current graphs that generate triangular embeddings of $K_{12s+6}-K_3$ and $K_{12s+9}-K_3$, for all $s \geq 2$. Kainen subembeddings are then constructed using Gross's method \cite{Gross-Case6} for finding triangular embeddings of $K_n-K_{1,3}$.

\begin{lemma}
When $n \equiv 6, 9 \Mod{12}$, $n \neq 9$, the surface crossing number of $K_n$ in the surface of genus $h(n)$ is 3. The surface crossing number of $K_9$ in $S_2$ is 4.
\label{lem-case69}
\end{lemma}
\begin{proof}
Harary and Hill \cite{HararyHill} and Riskin \cite{Riskin-Genus2} showed that the crossing numbers of $K_6$ in $S_0$ and $K_9$ in $S_2$ are 3 and 4, respectively. The triangular embedding of $K_{18}-3K_2$ in Sun \cite{Sun-Index2} is a Kainen subembedding of $K_{18}$. 

For all larger values of $n$, we consider the index 3 current graphs in Figures \ref{fig-case6gen}, \ref{fig-case9s1}, and \ref{fig-case9gen}. In all of their derived embeddings, the rotation at vertex $1$ is of the form 
$$\begin{array}{rlllllllllllllllllllllllllllll}
1. & \dots & x & u_1 & y & u_2 & z & \dots
\end{array}$$
By flipping the edges $(1, u_1) \to (x, y)$ and $(1, u_2) \to (y, z)$, the edge $(x, z)$ can be routed across the edge $(1, y)$. As depicted in Figure \ref{fig-case69flips}, the missing edges $(1, u_1)$ and $(1, u_2)$ can be reintroduced by having them cross over the edges $(v_1, v_2)$ and $(v_3, v_4)$, respectively, where the values $(u_1, u_2, v_1, v_2, v_3, v_4)$ are
\begin{itemize}
\item $(0, 2, 6, 4, 3, 12s)$, for $n = 12s+6$, $s \geq 2$. 
\item $(3, 8, 9, 12, 16, 11)$, for $n = 21$.
\item $(9s+3, 6s+5, 3s, 12s+4, 6s+6, 0)$, for $n = 12s+9$, $s \geq 2$. 
\end{itemize}
\end{proof}

\begin{figure}[tbp]
\centering
\includegraphics[scale=0.9]{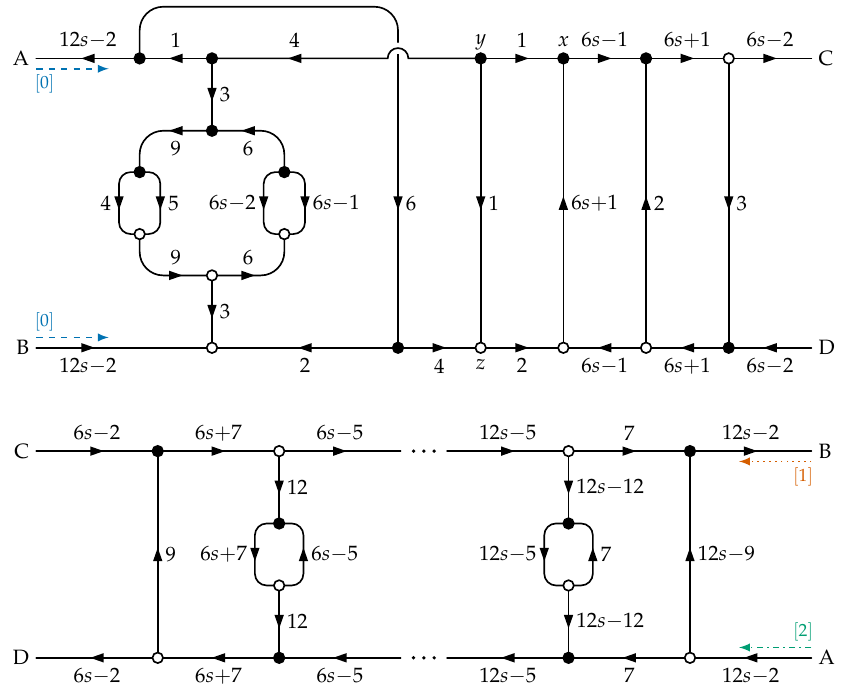}
\caption{A family of index 3 current graphs with current group $\mathbb{Z}_{12s+3}$.}
\label{fig-case6gen}
\end{figure}

\begin{figure}[tbp]
\centering
\includegraphics[scale=0.9]{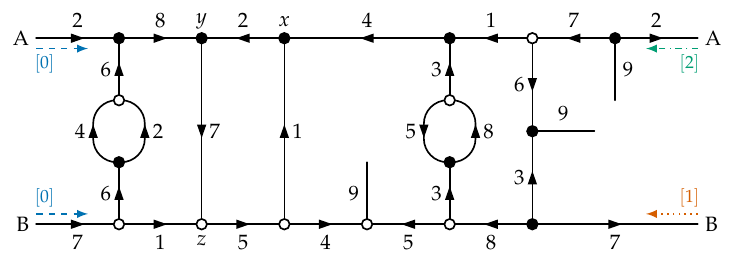}
\caption{Index 3 current graph with current group $\mathbb{Z}_{18}$.}
\label{fig-case9s1}
\end{figure}

\begin{figure}[tbp]
\centering
\includegraphics[scale=0.9]{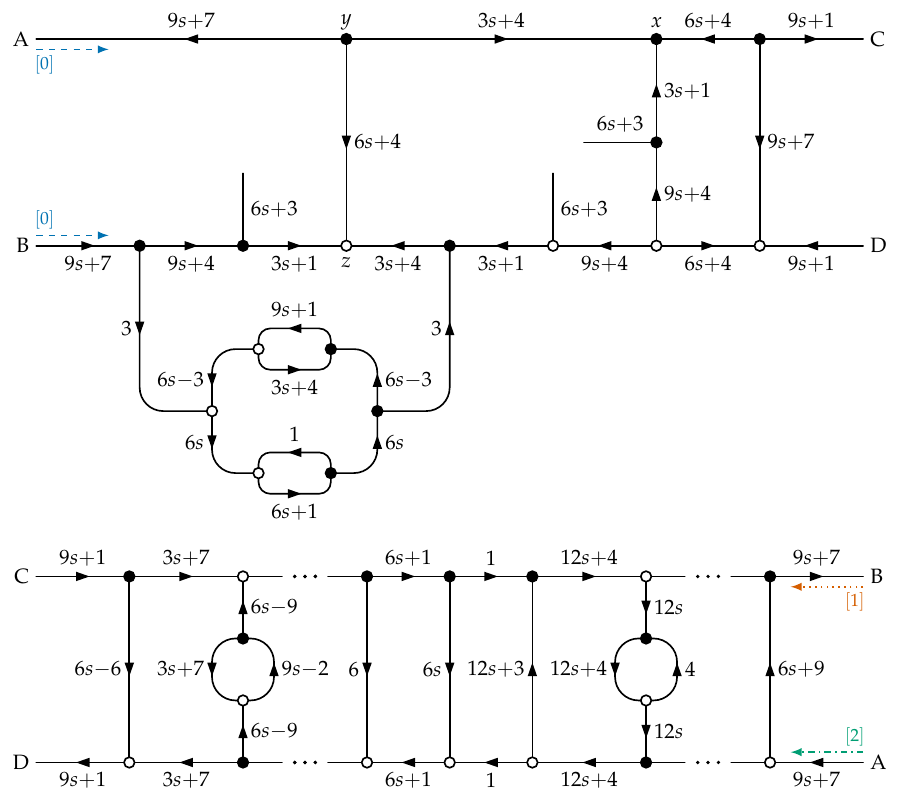}
\caption{Index 3 current graphs with current group $\mathbb{Z}_{12s+6}$.}
\label{fig-case9gen}
\end{figure}

\begin{figure}[tbp]
\centering
\includegraphics[scale=0.9]{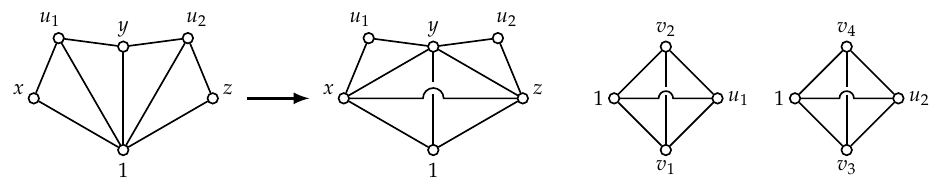}
\caption{Edge flips and crossings for $n \equiv 6, 9 \Mod{12}$.}
\label{fig-case69flips}
\end{figure}

There are several known families of index 1 current graphs (e.g., Sun \cite{Sun-FaceDist}) that generate triangular embeddings of $K_{12s+1}-K_3$, $s \geq 2$. Like in Gross's observation about the $n \equiv 10 \Mod{12}$ case, the vortices are too far apart, so we lift to index 2. Unfortunately, we encounter the same issue as in Lemma \ref{lem-case2}, where the vortices cannot be as close together as we had in Lemma \ref{lem-case69}. The workaround similarly uses further edge flips. 

\begin{lemma}
When $n = 12s+1$, $s \geq 1$, the surface crossing number of $K_n$ in the surface of genus $h(n)$ is 3.
\label{lem-case1}
\end{lemma}
\begin{proof}
Jungerman's \cite{Jungerman-K18} triangular embedding of $K_{13}-K_3$ is a Kainen subembedding of $K_{13}$. Consider the index 2 current graphs in Figures \ref{fig-case1-s2}, \ref{fig-case1-general-odd}, and \ref{fig-case1-general-even}. For $s = 2$, we start with the sequence of edge flips
$$(0, 15) \to (5, z) \to (12, 16) \to (0, 3) \to (x, y).$$
In the resulting embedding, we can apply one more edge flip $(0, 5) \to (y, z)$. Now, the missing edges $(x, z)$, $(0, 15)$, and $(0, 5)$ can be drawn into the embedding using the three crossings in Figure \ref{fig-case1-mod}(a). 

The general case follows a similar pattern. The rotation at vertex $0$ is of the form
$$\begin{array}{rlllllllllllllllllllllllllllll}
0. & \dots & x & u_1 & y & u_2 & u_3 & z & \dots
\end{array}$$
where $(u_1, u_2, u_3) = (3, 6s+3, 6s+7)$ or $(6s+1, 6s-1, 12s-5)$ for odd or even $s$, respectively. We start with the edge flip $(0, u_1) \to (x, y)$, and then the sequence
$$(0, u_2) \to (u_3, y) \to (v_1, v_2) \to (0, u_2),$$
where $(v_1, v_2) = (4, 6s+4)$ or $(6s-4, 6s-6)$ for odd or even $s$, respectively. After another flip $(0, u_3) \to (y, z)$, the three missing edges $(x, z)$, $(0, u_1)$, and $(0, u_3)$ can each be drawn with one crossing, following Figure \ref{fig-case1-mod}(b). The other vertices in the figure are $(w_1, w_2, w_3, w_4) = (1, 9s, 6, 6s+1)$ or $(1, 3s-1, 6s, 3s-2)$ for odd or even $s$, respectively. 
\end{proof}

\begin{figure}[tbp]
\centering
    \includegraphics[scale=0.9]{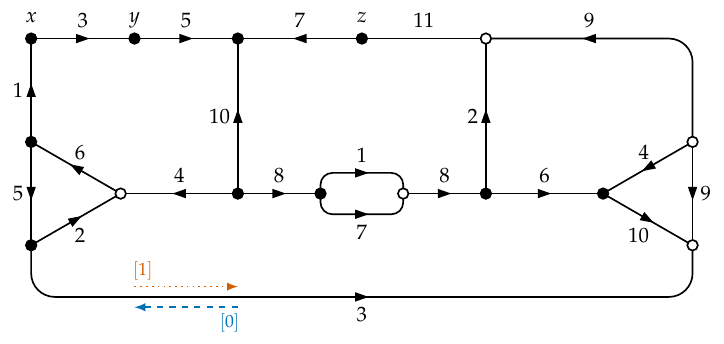}
\caption{Index 2 current graph for with current group $\mathbb{Z}_{22}$.}
\label{fig-case1-s2}
\end{figure}

\begin{figure}[tbp]
\centering
\includegraphics[width=\textwidth]{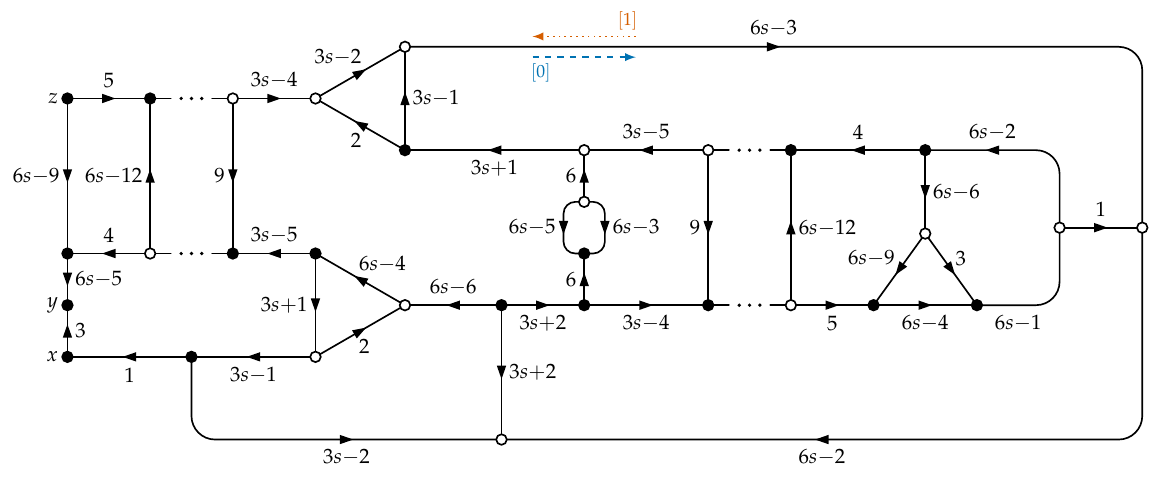}
\caption{Index 2 current graphs with current group $\mathbb{Z}_{12s-2}$, for odd $s \geq 3$.}
\label{fig-case1-general-odd}
\end{figure}

\begin{figure}[tbp]
\centering
\includegraphics[width=\textwidth]{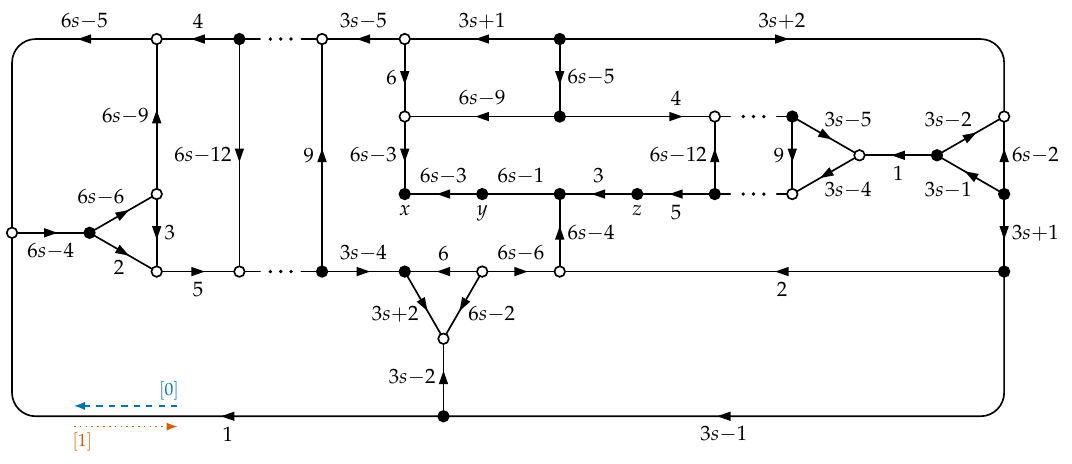}
\caption{Index 2 current graphs with current group $\mathbb{Z}_{12s-2}$, for even $s \geq 4$.}
\label{fig-case1-general-even}
\end{figure}

\begin{figure}[tbp]
\centering
    \begin{subfigure}[b]{0.99\textwidth}
        \centering
        \includegraphics[scale=0.9]{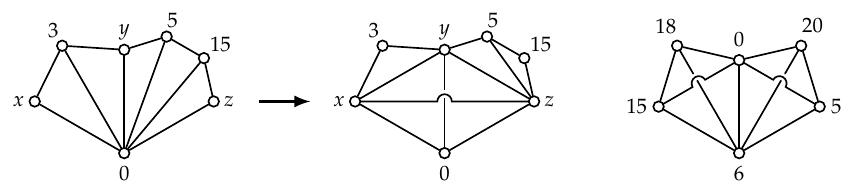}
        \caption{}
    \end{subfigure}
    \begin{subfigure}[b]{0.99\textwidth}
        \centering
        \includegraphics[scale=0.9]{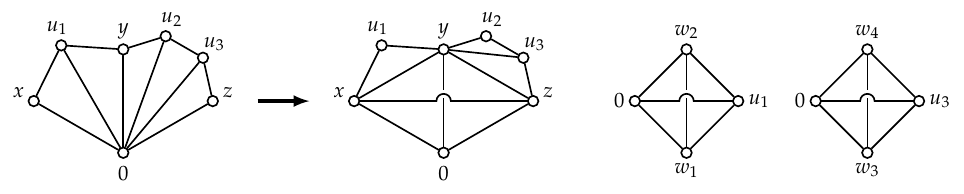}
        \caption{}
    \end{subfigure}
\caption{Edge flips and crossings for $K_{12s+1}$, for $s = 2$ (a) and $s \geq 3$ (b).}
\label{fig-case1-mod}
\end{figure}

The remaining case where Kainen drawings have three crossings is the dreaded $n \equiv 10 \Mod{12}$ case. Our \emph{ad hoc} approach to the $n = 22$ case uses a vortex that does not induce a Hamiltonian face:
\begin{enumerate}
\item[(V4)] In an index 3 current graph, labeled vertex of degree $3$, which is incident with each circuit, and whose excess has order 3 in $\mathbb{Z}_{18}$.
\end{enumerate}

To make use of higher index current graphs, our solution for the $n \equiv 10 \Mod{12}$ case starts with triangular embeddings of $K_{12s+10}-K_6$. Triangular embeddings of $K_n-K_6$ were first considered by Guy and Ringel \cite{GuyRingel}, who showed that if the vortices are positioned correctly, two handles can be used to obtain a triangular embedding of $K_n-K_{1,3}$. Our handle operations are slightly different, since we have the luxury of introducing crossings. 

\begin{lemma}
When $n = 12s+10$, $s \geq 0$, the surface crossing number of $K_n$ in the surface of genus $h(n)$ is 3.
\label{lem-case10}
\end{lemma}
\begin{proof}
The smallest case where $s = 0$ was previously discussed in Section \ref{sec-general}. For $s = 1$, we require an unusual construction using the current graph in Figure \ref{fig-case10s1}. In the derived embedding, the vortex labeled $c$ induces two 9-sided faces, which we label $c_0$ and $c_1$, incident with exactly the even and odd vertices, respectively. After adding a handle near vertex 0, as seen in Figure \ref{fig-k22-handle}(a), we flip the edges 
$$(2, b) \to (0, 13) \to (1, 5)$$
to restore the deleted edge $(1, 5)$. Inside of the large face shown in Figure \ref{fig-k22-handle}(b), an edge between vertices $c_0$ and $c_1$ can be added and contracted, resulting in a new vertex $c$ adjacent to all of the numbered vertices. At this point, all remaining missing edges can be added inside of the large face or as crossings elsewhere in the embedding. 

\begin{figure}[tbp]
\centering
\includegraphics[scale=0.9]{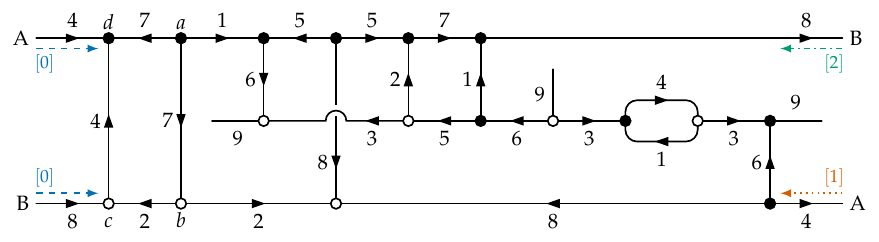}
\caption{Index 3 current graph with group $\mathbb{Z}_{18}$.}
\label{fig-case10s1}
\end{figure}

\begin{figure}[tbp]
\centering
\begin{subfigure}[b]{0.99\textwidth}
    \centering
    \includegraphics[scale=0.9]{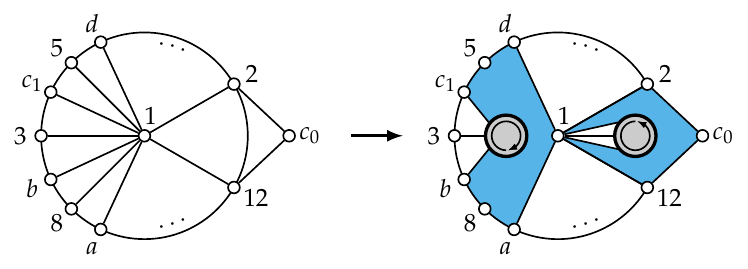}
    \caption{}
\end{subfigure}
\begin{subfigure}[b]{0.99\textwidth}
    \centering
    \includegraphics[scale=0.9]{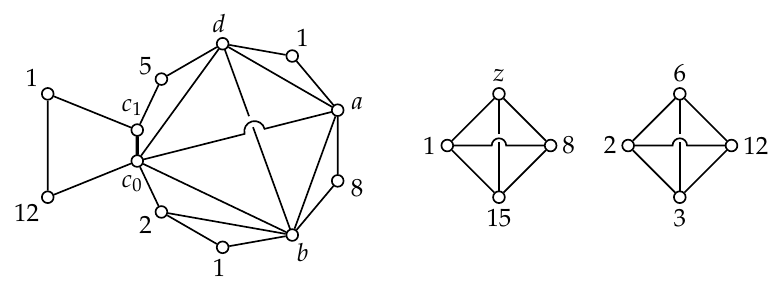}
    \caption{}
\end{subfigure}
\caption{A handle for $K_{22}$.}
\label{fig-k22-handle}
\end{figure}

For $s \geq 2$, the rotation at vertex 0 is of the form
$$\begin{array}{rlllllllllllllllllllllllllllll}
0. & \dots & a & v_1 & b & \dots & c & v_2 & d & \dots & e & v_3 & f & \dots
\end{array}$$
where $(v_1,v_2,v_3)$ is $(6s-3, 6s+7, 1)$ for $s$ even and $(12s-1, 3s+4, 6s+3)$ for $s$ odd. The modification in Figure \ref{fig-case10handle}(a) to the embedding near vertex $0$ deletes the edges $(0, v_1)$, $(0, v_2)$, and $(0, v_3)$, and then merges the resulting three quadrilaterals. When $s$ is even, we apply the flips
$$(v_3, a) \to (6s+2, 6s+8) \to (0, v_2)$$
and 
$$(v_3, d) \to (6s-2, 6s) \to (0, v_1).$$
When $s$ is odd, we apply the flips
$$(v_3, a) \to (6s+2, 6s+8) \to (0, v_1)$$
and
$$(v_3, d) \to (3s-1, 3s+1) \to (0, v_2).$$
In the merged face, we add the edges in Figure \ref{fig-case10handle}(b). By connecting the two shaded faces with a triangle, the resulting handle in Figure \ref{fig-case10handle}(c) can be used to add all nine missing edges between the lettered vertices using three crossings. 

\begin{figure}[tbp]
\centering
\includegraphics[scale=0.9]{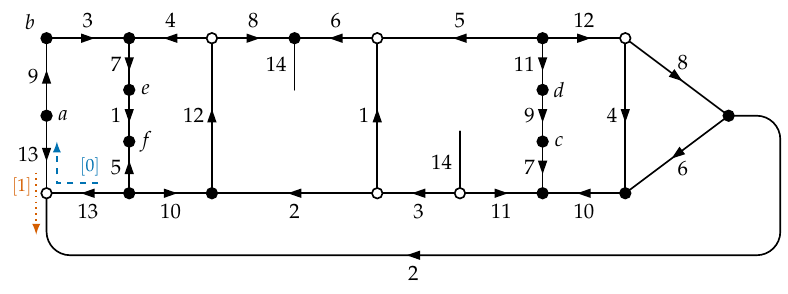}
\caption{Index 2 current graph with group $\mathbb{Z}_{28}$.}
\label{fig-case10s2}
\end{figure}

\begin{figure}[tbp]
\centering
\includegraphics[width=\textwidth]{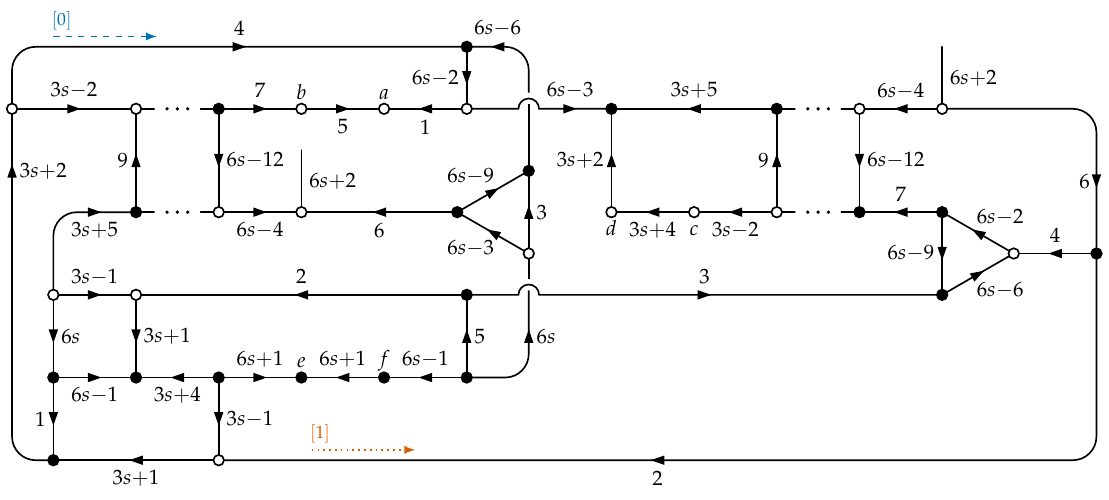}
\caption{Index 2 current graphs with current group $\mathbb{Z}_{12s+4}$, for odd $s \geq 3$.}
\label{fig-current-c10-odd}
\end{figure}

\begin{figure}[tbp]
\centering
\includegraphics[width=\textwidth]{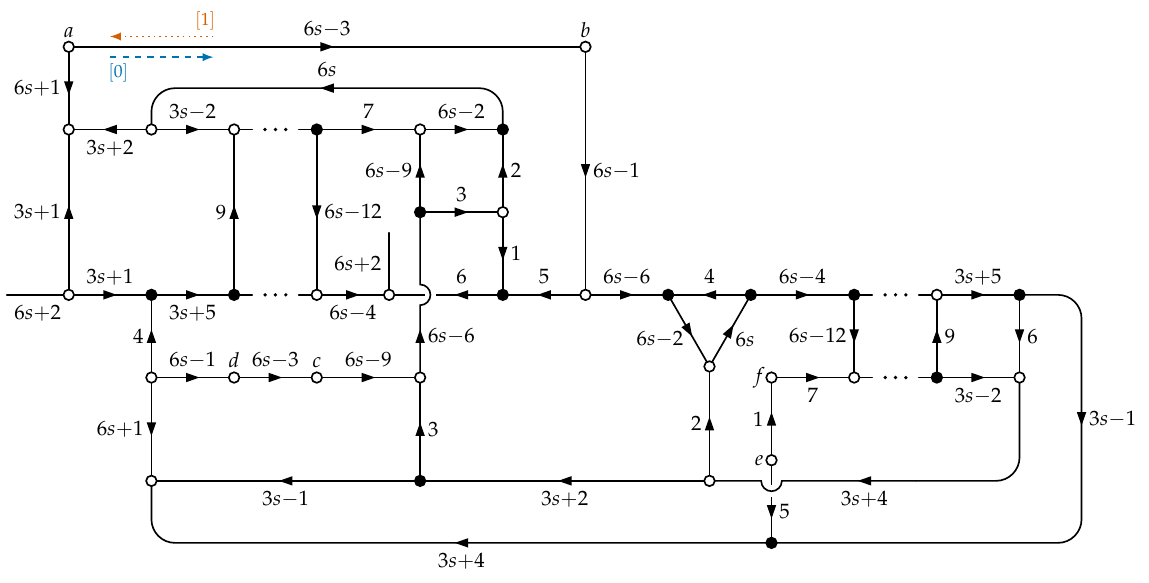}
\caption{Index 2 current graphs with current group $\mathbb{Z}_{12s+4}$, for even $s \geq 4$}
\label{fig-current-c10-even}
\end{figure}

\begin{figure}[tbp]
\centering
    \begin{subfigure}[b]{0.99\textwidth}
        \centering
        \includegraphics[scale=0.9]{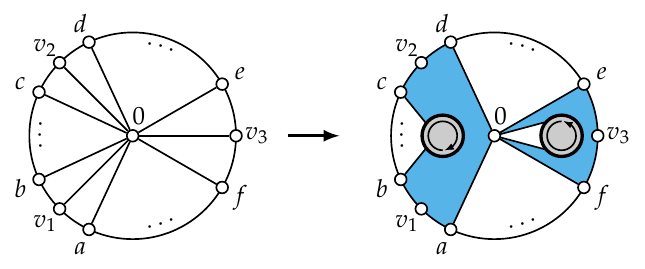}
        \caption{}
    \end{subfigure}
    \begin{subfigure}[b]{0.33\textwidth}
        \centering
        \includegraphics[scale=0.9]{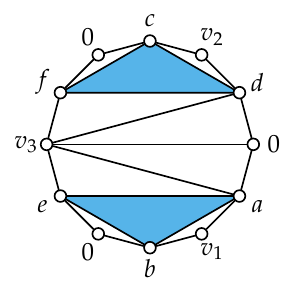}
        \caption{}
    \end{subfigure}
    \begin{subfigure}[b]{0.33\textwidth}
        \centering
        \includegraphics[scale=0.9]{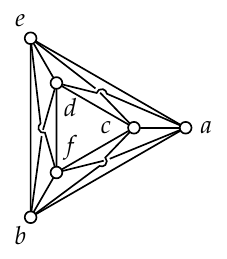}
        \caption{}
    \end{subfigure}
\caption{Adding all the edges between vortices using two handles.}
\label{fig-case10handle}
\end{figure}
\end{proof}

\begin{lemma}
When $n = 12s+8$, $s \geq 0$, the surface crossing number of $K_n$ in the surface of genus $h(n)$ is 4.
\label{lem-case8}
\end{lemma}
\begin{proof}
Guy \emph{et al.} \cite{Guy-ToroidalComplete} found a drawing of $K_8$ in the torus with four crossings. For $s = 1$, starting with the rotation system in Table \ref{tab-k20} in Appendix \ref{app-table}, the modification described in Figure \ref{fig-k20add} deletes a path on 3 edges to merge $w_0$ and $w_1$ into a single vertex $w$. The edge $(w,x)$ and the edges of the deleted path can each be drawn in using one additional crossing. In particular, $(0,1)$ crosses $(2, 6)$, $(0, 16)$ crosses $(4, 5)$, and $(1, 17)$ crosses $(4, 13)$.  

\begin{figure}[tbp]
\centering
\includegraphics[scale=0.9]{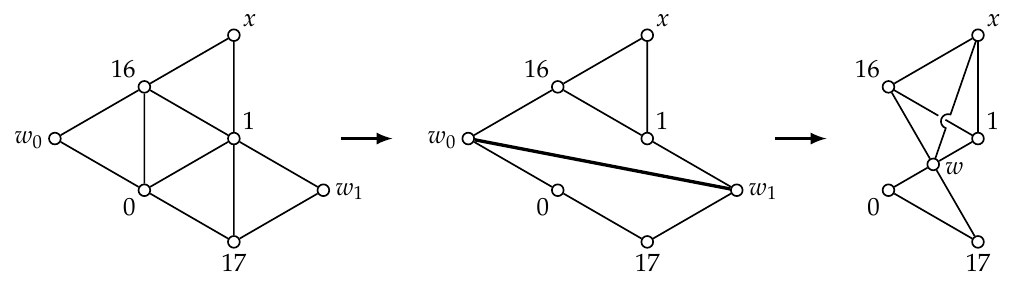}
\caption{Constructing a Kainen drawing of $K_{20}$.}
\label{fig-k20add}
\end{figure}

Consider the family of current graphs in Figure \ref{fig-case8gen} defined for all $s \geq 2$. The derived embeddings are triangular embeddings of $K_{12s+8}-K_5$, and the rotation at vertex $1$ is of the form
$$\begin{array}{rlllllllllllllllllllllllllllll}
1. & \dots & a & 0 & b & 6s-1 & 6 & c & 12s & d & 6s+5 & e & \dots
\end{array}$$
The goal of the initial modification in Figure \ref{fig-case8flips} is to delete enough edges so that the edges $(a, b)$ and $(c, e)$ can be added. After applying two sequences of edge flips 
$$(6s+5, b) \to (3, 7) \to (c, d)$$ 
and 
$$(12s, a) \to (12s-1, 12s+1) \to (0,1),$$ 
all remaining edges, except $(d, e)$, can be added inside of a handle connecting the two shaded faces in Figure \ref{fig-case8flips}(a), as seen in Figure \ref{fig-case8flips}(b). There are several locations to incorporate the last edge $(d,e)$ using one crossing, since in the original derived embedding, $d$ and $e$ are separated by one vertex in the rotation of every numbered vertex that is not a multiple of 3. 
\end{proof}

\begin{figure}[tbp]
\centering
\includegraphics[scale=0.9]{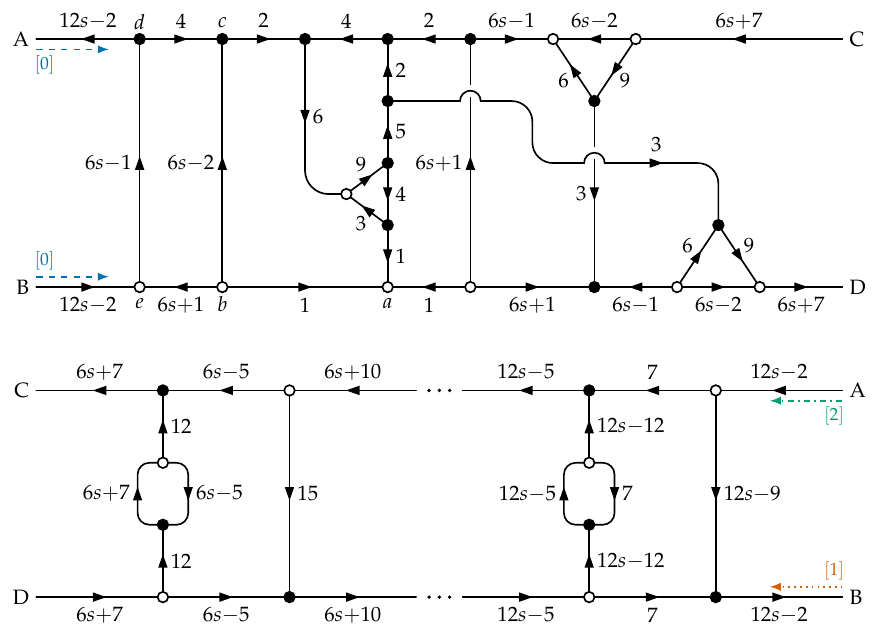}
\caption{A family of index 3 current graphs with current group $\mathbb{Z}_{12s+3}$.}
\label{fig-case8gen}
\end{figure}

\begin{figure}[tbp]
\centering
\begin{subfigure}[b]{0.99\textwidth}
    \centering
    \includegraphics[scale=0.9]{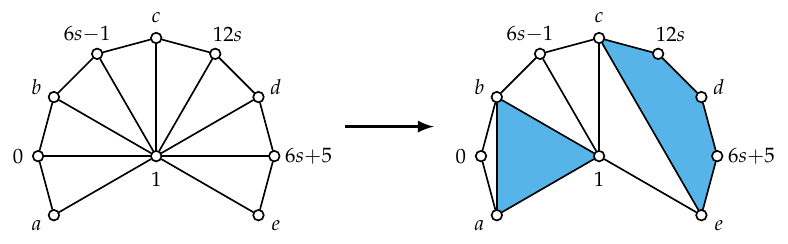}
    \caption{}
\end{subfigure}
\begin{subfigure}[b]{0.99\textwidth}
    \centering
    \includegraphics[scale=0.9]{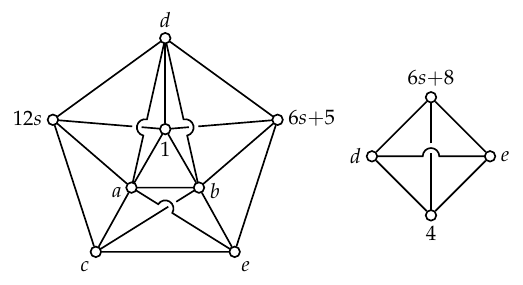}
    \caption{}
\end{subfigure}
\caption{A handle for $n \equiv 8 \Mod{12}$.}
\label{fig-case8flips}
\end{figure}

Surprisingly, the general solution for $n \equiv 11 \pmod{12}$ is tackled using index 1 current graphs, as was done in the original Map Color Theorem by Ringel and Youngs \cite{RingelYoungs-Case11}. They introduced the final type of vortex we need:

\begin{enumerate}
\item[(V5)] In an index 1 current, a labeled vertex of degree $3$, where the excess of the vertex generates the subgroup consisting of multiples of 3, and the incident currents are congruent modulo 3.
\end{enumerate}

Such a vortex generates three Hamiltonian faces, so we label each face corner of the vortex with a different letter. 

\begin{lemma}
When $n = 12s+11$, $s \geq 0$, the surface crossing number of $K_n$ in the surface of genus $h(n)$ is 4.
\label{lem-case11}
\end{lemma}
\begin{proof}
In an unpublished work, Lee \cite{Lee-GenusCrossing} found a Kainen drawing of $K_{11}$, which is rendered as a rotation system of a Kainen subembedding in Appendix \ref{app-table}. The triangular embedding of $K_{23}-K_{1,4}$ constructed in Sun \cite{Sun-Minimum} is a Kainen subembedding of $K_{23}$. 

Consider the current graphs in Figure \ref{fig-case11s2}, \ref{fig-case11s3}, and \ref{fig-case11sgen}. In each derived embedding, which is a triangular embedding of $K_n-K_5$, the rotation at vertex $0$ is of the form
$$\begin{array}{rlllllllllllllllllllllllllllll}
0. & \dots & a & 6s+1 & x & 6s+5 & b & \dots & y & v & 9s+8 & c & \dots,
\end{array}$$
where $v$ is $19$, $25$, and $1$ for $s = 2$, $s = 3$, and $s \geq 4$, respectively. In all cases, there is a sequence of edge flips that starts with $(v, c)$ and ends with $(0, 6s+5)$:
\begin{itemize}
\item for $s = 2$, $(19, c) \to (2, 27) \to (0, 17)$,
\item for $s = 3$, $(25, c) \to (2, 36) \to (0, 23)$, and
\item for $s \geq 4$: $(1, c) \to (6s+2, 9s+3) \to (5, 6s+7) \to (0, 6s+5)$.
\end{itemize}
The same part of the current graph that enables the last edge flip in each of these sequences also allows the edge $(6s+1, 0)$ to be drawn into the embedding with one crossing. 

\begin{figure}[tbp]
\centering
\includegraphics[scale=0.9]{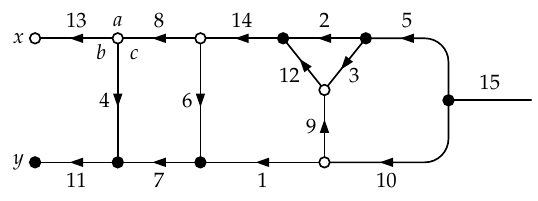}
\caption{Index 1 current graph with group $\mathbb{Z}_{30}$.}
\label{fig-case11s2}
\end{figure}

\begin{figure}[tbp]
\centering
\includegraphics[scale=0.9]{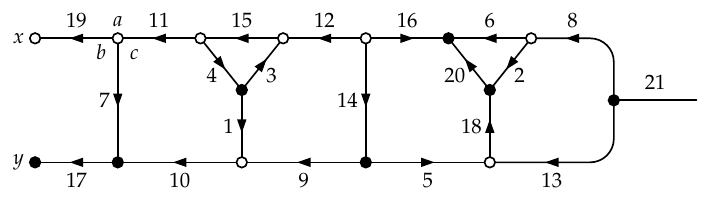}
\caption{Index 1 current graph with group $\mathbb{Z}_{42}$.}
\label{fig-case11s3}
\end{figure}

\begin{figure}[tbp]
\centering
\includegraphics[scale=0.9]{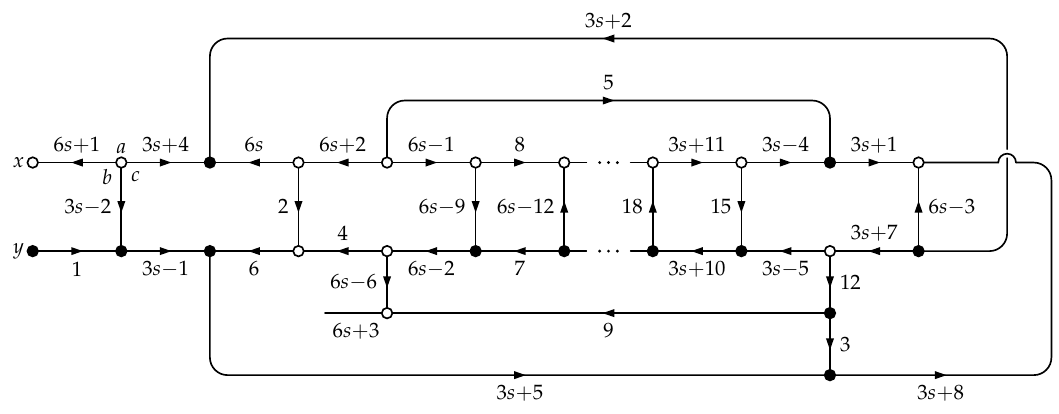}
\caption{A family of index 1 current graphs with group $\mathbb{Z}_{12s+6}$, $s \geq 4$.}
\label{fig-case11sgen}
\end{figure}

After connecting the shaded faces in Figure \ref{fig-case11-handle}(a) with a handle, Figure \ref{fig-case11-handle}(b) shows how to connect all the lettered vertices using only one crossing. The remaining two edges $(0, v)$ and $(0, 9s+8)$ can both be drawn into the embedding using one crossing, as seen in Figure \ref{fig-case11-flips}. 
\end{proof}

\begin{figure}[tbp]
\centering
\begin{subfigure}[b]{0.65\textwidth}
    \centering
    \includegraphics[scale=0.9]{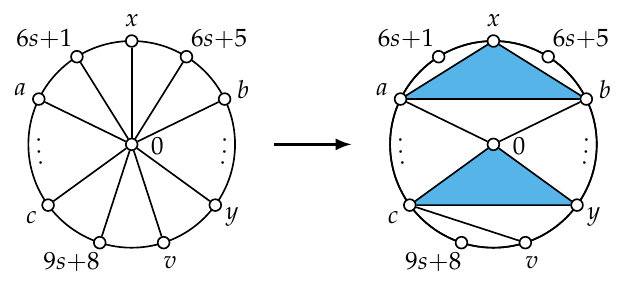}
    \caption{}
\end{subfigure}
\begin{subfigure}[b]{0.34\textwidth}
    \centering
    \includegraphics[scale=0.9]{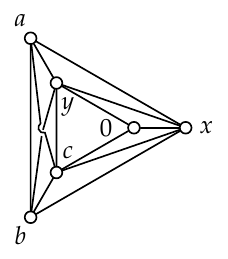}
    \caption{}
\end{subfigure}
\caption{A handle for $n \equiv 11 \Mod{12}$.}
\label{fig-case11-handle}
\end{figure}

\begin{figure}[tbp]
\centering
    \begin{subfigure}[b]{0.32\textwidth}
        \centering
        \includegraphics[scale=0.9]{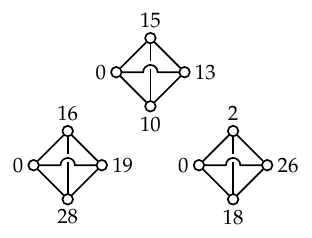}
        \caption{}
    \end{subfigure}
    \begin{subfigure}[b]{0.32\textwidth}
        \centering
        \includegraphics[scale=0.9]{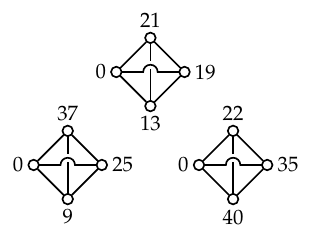}
        \caption{}
    \end{subfigure}
    \begin{subfigure}[b]{0.32\textwidth}
        \centering
        \includegraphics[scale=0.9]{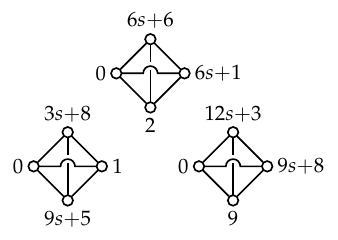}
        \caption{}
    \end{subfigure}
\caption{Crossings for $K_{12s+11}$, $s = 2$ (a), $s = 3$ (b), and $s \geq 4$ (c).}
\label{fig-case11-flips}
\end{figure}

\section{Kainen drawings of $K_{m,n}$}\label{sec-kmn}

\subsection{$K_{5,5}$ needs two crossings}\label{sec-k55}

As mentioned earlier, the new bipartite counterexample to Kainen's conjecture is $K_{5,5}$. Kainen's lower bound predicts that in the surface $S_2$, $K_{5,5}$ can be drawn with one crossing. The proof is a matter of case-checking, relying on immediate properties of quadrilateral embeddings in orientable surfaces. In this section, rotations and face boundary walks are given clockwise and counterclockwise orientations, respectively.

\begin{property}[orientability]
For every edge $(u, v)$, there is exactly one quadrilateral of the form $[\dotsc, u, v, \dotsc]$ and one of the form $[\dotsc, v, u, \dotsc].$ 
\label{prop-orient}
\end{property}

\begin{property}[manifold]
If the rotation at vertex $u$ is of the form 
$$\begin{array}{rrrrrrrrrrrr}
u. & v_1 & v_2 & v_3 & \dotsc & v_j, \\ 
\end{array}$$
then there are quadrilateral faces of the form $$[v_1, u, v_2, \dotsc], [v_2, u, v_3, \dotsc], \dotsc, [v_j, u, v_1, \dotsc].$$
\label{prop-manifold}
\end{property}

\begin{theorem}
The crossing number of $K_{5,5}$ in $S_2$ is $2$.
\end{theorem}
\begin{proof}
Let the left vertices of $K_{5,5}$ be labeled $a$, $b$, $c$, $d$, and $e$, and let the right vertices be labeled $0$, $1$, $2$, $3$, and $4$. Suppose there exists a Kainen subembedding of $K_{5,5}$ that is missing the edge $(a,0)$. Without loss of generality, we may assume that the embedding near vertices $a$ and $0$ looks like Figure~\ref{fig-k55initial}, in which there are adjacent faces $\lbrack 0, b, 1, e\rbrack$ and $\lbrack a, 1, b, 4\rbrack$ and the missing edge can be drawn in by crossing over the edge $(b,1)$. By the Euler polyhedral equation, any quadrilateral embedding of $K_{5,5}-K_2$ must have 12 faces. Thus, there are four additional faces besides those incident with $a$ and $0$ missing from Figure~\ref{fig-k55initial}.

\begin{figure}[tbp]
\centering
\includegraphics[scale=0.9]{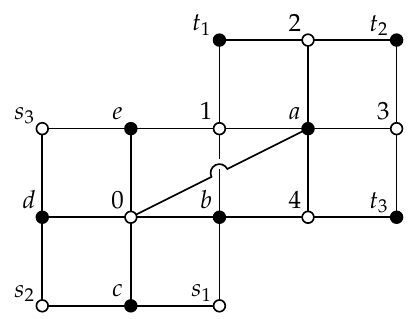}
\caption{A submap of a purported 1-crossing drawing of $K_{5,5}$.}
\label{fig-k55initial}
\end{figure}

Since both vertices $4$ and $b$ have degree $5$, the vertices labeled $s_1$ and $t_3$ cannot be $4$ or $b$, respectively, otherwise there would be a manifold violation. By orientability, there must be a face of the form $[4, b, \dotsc]$, which is one of the four missing faces. Of the unknown vertices, the only other vertex that can be $b$ is $t_2$. Similarly, there is another new face of the form $[e, 1, \dotsc]$ (since each face boundary alternates between left and right vertices, it cannot be the same as the previously discovered face), and the only other potential instance of vertex $1$ is $v_1$. We distinguish between two cases: either $b = t_2$ or $1 = s_2$, or neither is true.

In the former case, suppose without loss of generality that vertex $b = t_2$, as in Figure \ref{fig-k55case1}(a). There are now five quadrilaterals incident with $b$, allowing us to determine that $s_1 = 2$ and $s_0 = 3$ using the manifold property. We next determine the identities of vertices $t_0$ and $t_1$. $t_1$ cannot be $c$, because otherwise there would be two faces of the form $[c, 2, \dotsc]$, an orientability violation. Thus, $t_0 = c$, and $t_1 = d$. Since there must be quadrilaterals of the form $[c, 1, \dotsc]$ and $[d, 2, \dotsc]$, those must be the remaining two missing faces. By the manifold property with $u = d$, these faces are distinct. To satisfy the manifold property with $u = 2$, there must be a face of the form $[e, 2, c, \dotsc]$. However, no such face fits this description: only one face is incident with both $c$ and $e$, but it would have the wrong orientation.

\begin{figure}[tbp]
\centering
    \begin{subfigure}[b]{0.49\textwidth}
        \centering
        \includegraphics[scale=0.9]{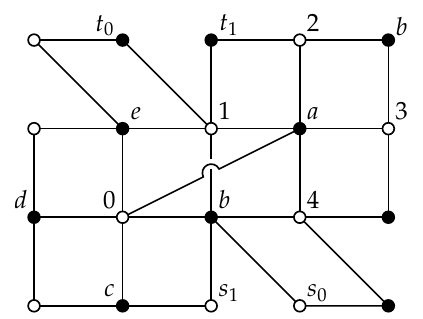}
        \caption{}
    \end{subfigure}
    \begin{subfigure}[b]{0.49\textwidth}
        \centering
        \includegraphics[scale=0.9]{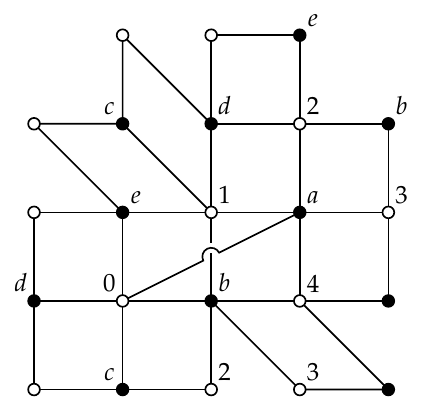}
        \caption{}
    \end{subfigure}
\caption{The repeated vertex case.}
\label{fig-k55case1}
\end{figure}

When $b \neq t_2$ and $1 \neq s_2$, then there are two new faces incident with each of vertices $1$ and $b$. These pairs of faces are disjoint because the edge $(1,b)$ is already incident with two faces. In Figure~\ref{fig-k55case2}, all twelve faces are shown. All vertices with label $s_i$ (respectively, $t_i$) must be one of the vertices $2, 3, 4$ (respectively, $c, d, e$). Every vertex except $a$ and $0$ has degree $5$, which allows us to infer the following constraints:

\begin{itemize}
\item $\{s_0, s_1\} = \{2, 3\}$
\item $\{t_0, t_1\} = \{c, d\}$. 
\item Exactly two of $s_2, \dotsc, s_6$ is vertex $4$.
\item Exactly two of $t_2, \dotsc, t_6$ is vertex $e$. 
\end{itemize}

We consider which combinations of vertices correspond to vertex $e$. They cannot be consecutive vertices $t_i, t_{i+1}$, otherwise there is a manifold violation. Vertices $t_2$ and $t_5$ cannot both be $e$, since the fact that $\{s_0, s_1\} = \{2, 3\}$ forces either a manifold or orientability violation, Thus, either $t_2 = t_4 = e$ or $t_3 = t_5 = e$, and similarly, $s_2 = s_4 = e$ or $s_3 = s_5 = 4$. 

If $t_2 = t_4 = e$ and $s_3 = s_5 = 4$, or $t_3 = t_5 = e$ and $s_2 = s_4 = 4$, then there is an orientability violation at the edge $(4, e)$. If $t_2 = t_4 = e$ and $s_2 = s_4 = 4$, then by Property \ref{prop-manifold} on $u = e$, $s_0 = 2$, forcing $s_1 = 3$. By symmetry, $t_0 = c$ and $t_1 = d$. Since the edge $(2,e)$ is already incident with two faces, $s_3 = 3$, forcing $s_5 = 2$. However, there are two faces of the form $[1, d, 2, \dotsc]$ and $[2, d, 1, \dotsc]$, which is a manifold violation at $u = d$. Finally, if both $t_3 = t_5 = e$ and $s_3 = s_5 = 4$, then by the manifold property applied to $u = e$, $s_0 = s_4 = 2$ and $s_1 = 3$. Since there are already five faces incident with vertex $2$, $s_2$ must be $3$, which would be a manifold violation at vertex $u = c$. 

\begin{figure}[tbp]
\centering
\includegraphics[scale=0.9]{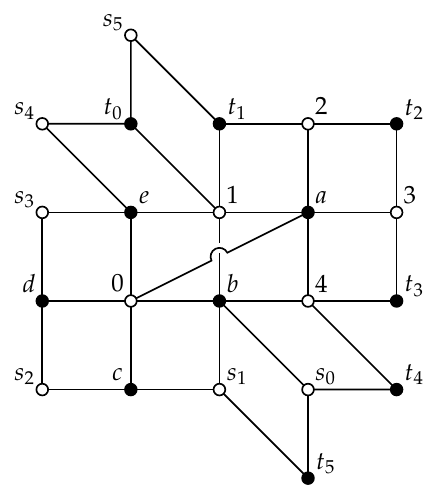}
\caption{The distinct faces case.}
\label{fig-k55case2}
\end{figure}

Any quadrilateral embedding of $K_{5,5}-e$, such as the one constructed by Mohar \emph{et al.} \cite{MPP-NearlyComplete}, can be used to give a drawing with at most two crossings. Let $m$ be any right vertex that shares a face with $a$. The edge $(0,a)$ can be added by crossing at most two edges incident with $m$, as seen in Figure \ref{fig-k55}. 
\end{proof}

\begin{figure}[tbp]
\centering
\includegraphics[scale=0.9]{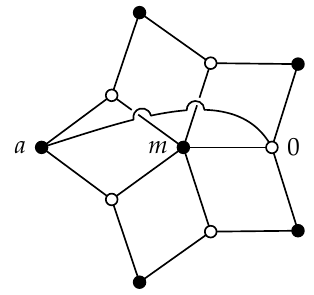}
\caption{$K_{5,5}$ in $S_2$ needs at most two crossings.}
\label{fig-k55}
\end{figure}

\subsection{The diamond sum operation}

Let $\psi\colon G \to S$ and $\psi'\colon G' \to S'$ be two quadrangular embeddings, and let $v$ and $v'$ be vertices of $G$ and $G'$, respectively, that have the same degree. Following Figure \ref{fig-diamond}(a), let $D$ be a closed disk in the surface $S$ intersecting the embedded graph only at $v$ and its incident edges, such that the neighbors of $v$ are on the boundary of $D$. Let $D'$ be defined similarly for $S'$ and $v'$. Delete the interiors of $D$ and $D'$ and identify the resulting boundaries, like in Figure \ref{fig-diamond}(b), so that each vertex on the boundary of $D$ is identified with a unique vertex on the boundary of $D'$. We call the resulting quadrangular embedding a \emph{diamond sum} of $\psi$ and $\psi'$ at $v$ and $v'$. Since the diamond sum operation takes the connected sum of the two surfaces, the resulting embedding is orientable if and only if both of the original embeddings are orientable. 

\begin{figure}[tbp]
\centering
    \begin{subfigure}[b]{0.55\textwidth}
    \centering
        \includegraphics[scale=1]{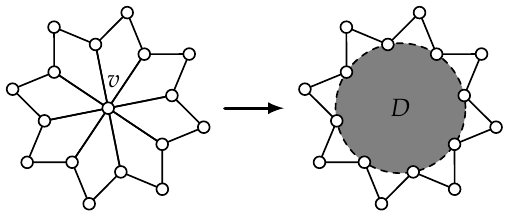}
        \caption{}
    \end{subfigure}
    \begin{subfigure}[b]{0.44\textwidth}
    \centering
        \includegraphics[scale=1]{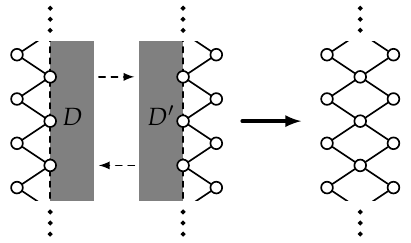}
        \caption{}
    \end{subfigure}
\caption{Excising disks (a) and merging their boundaries (b).}
\label{fig-diamond}
\end{figure}

The diamond sum operation is particularly suited for constructing minimum genus embeddings of complete bipartite graphs, but we can extend it to Kainen drawings, as well. Diamond sums are still possible as long as we can excise disks away from the crossings:

\begin{lemma}
Suppose that for $i = 1,2$, there exist Kainen drawings of $K_{m,n_i}$ with $k_i$ crossings in the surface of genus $g_i$, and both drawings have a right slacker. Then there is a Kainen drawing of $K_{m,n_1+n_2-2}$ with $k_1+k_2$ crossings in the surface of genus $g_1+g_2$.
\label{lem-diamondsum}
\end{lemma}
\begin{proof}
For each drawing, a disk can be drawn around the neighborhood of a right slacker that avoids any edges with nonzero responsibility. Then, temporarily delete one of the edges of each crossing to obtain Kainen subembeddings of $K_{m,n_1}$ and $K_{m,n_2}$ in the same surfaces. Taking the diamond sum of these two embeddings by excising the aforementioned disks results in a Kainen subembedding of $K_{m,n_1+n_2-2}$ with $k_1+k_2$ missing edges. Since each deleted edge avoids the excised disks, it can be returned to the drawing using one crossing each, yielding a Kainen drawing. 
\end{proof}

Sometimes, the existence of the desired slacker depends on the specific drawing, but it is guaranteed when there are few crossings (recall that Kainen drawings of $K_{m,n}$ in $S_{h(m,n)}$ have at most 3 crossings) compared to the number of vertices:

\begin{proposition}
Suppose we have a drawing of $K_{m,n}$ with at most $k$ crossings. If $m > 2k$ (respectively, $n > 2k$), then there is a left (respectively, right) slacker.
\label{prop-pigeon}
\end{proposition}
\begin{proof}
Each crossing rules out two vertices from each side, so in total, at most $2k$ vertices per side cannot be slackers.
\end{proof}

The proof relies on the same inductive step used by Bouchet \cite{Bouchet-Diamond}:

\begin{corollary}
If there exists a Kainen drawing of $K_{m,n}$, where $m \geq 3$ and $n \geq 6$, then there exists a Kainen drawing of $K_{m,n+4}$.
\label{cor-induct}
\end{corollary}
\begin{proof}
Note that Kainen drawings of $K_{m,n}$ and $K_{m,n+4}$ would have the same number of crossings, and that any Kainen drawing of $K_{m,n}$ has a right slacker by Proposition \ref{prop-pigeon} (if $n = 6$, then the drawing is an embedding). Apply Lemma \ref{lem-diamondsum} on the drawing of $K_{m,n}$ and a quadrangular embedding of $K_{m,6}$.  
\end{proof}

\subsection{The constructions}

Our proof of Theorem \ref{thm-kmn} begins with the subfamily $K_{3,n}$, which was already solved by Richter and \v{S}ir\'a\v{n}~\cite{RichterSiran-K3n}. We apply their construction to obtain Kainen drawings of $K_{3,n}$ and $K_{4,n}$, taking care to ensure the existence of a left slacker for each drawing of $K_{3,n}$ for future diamond sums. 

\begin{lemma}
For all $n \neq 5$, there exists a Kainen drawing of $K_{3,n}$ with a left slacker. Furthermore, when $n \geq 6$, the drawing also has a right slacker. 
\label{lem-k3n}
\end{lemma}
\begin{proof}
We note that, for $k \geq 0$, the following relationship holds between consecutive values of the functions $h(m,n)$ and $H(m,n)$:
$$H(3, 4k+2) = h(3, 4k+2) = h(3, 4k+3) = h(3, 4k+4) = h(3, 4k+5).$$
One way of achieving the desired drawings on the surface of this genus is to start with a quadrilateral embedding of $K_{3,4k+2}$ and ``double'' up to three right vertices, as illustrated in Figure~\ref{fig-doubling}(a). There are enough right vertices to double except in the case of $K_{3,5}$ (which would have been built from $K_{3,2}$).

The doubled vertices can each be placed so that some left vertex, like $c$ in Figure \ref{fig-doubling}, ends up with no responsibility. That vertex is our left slacker, and when $n \geq 6$, some right vertex is a slacker by Proposition \ref{prop-pigeon}.
\end{proof}

\begin{figure}[tbp]
\centering
    \begin{subfigure}[b]{0.49\textwidth}
    \centering
        \includegraphics[scale=0.9]{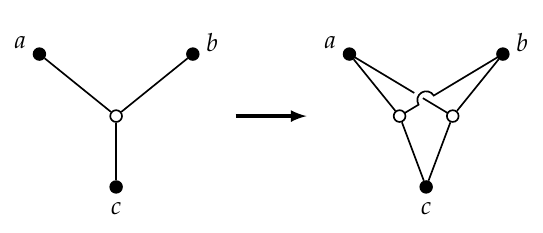}
        \caption{}
    \end{subfigure}
    \begin{subfigure}[b]{0.49\textwidth}
    \centering
        \includegraphics[scale=0.9]{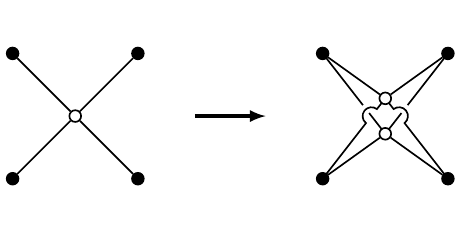}
        \caption{}
    \end{subfigure}
\caption{Doubling vertices of degree 3 (a) and 4 (b).}
\label{fig-doubling}
\end{figure}

Doubling can be performed on vertices of arbitrary degree, but if we wish to have responsibility at most 1 for each edge, then the only other application is the 4-valent case:

\begin{lemma}
There are Kainen drawings of $K_{4,n}$ for all $n \geq 2$. 
\label{lem-k4n}
\end{lemma}
\begin{proof}
For even $n \geq 2$, $K_{4,n}$ has a quadrilateral embedding. We can obtain a Kainen drawing of $K_{4,n+1}$ with two crossings by the doubling operation shown in Figure \ref{fig-doubling}(b).
\end{proof}

For all remaining complete bipartite graphs, we find Kainen drawings by repeatedly applying the diamond sum operation. 

\begin{lemma}
There are Kainen drawings of $K_{5,n}$ for all $n \geq 6$. 
\label{lem-k5n}
\end{lemma}
\begin{proof}
We induct on $n$ using Corollary \ref{cor-induct}, with base cases $n = 6, \dotsc, 9$. The graph $K_{5,6}$ has a quadrilateral embedding. Kainen drawings of $K_{5,7}$ and $K_{5,9}$ are shown in Figures \ref{fig-k57} and \ref{fig-k59}. A Kainen drawing of $K_{5,8}$ with two crossings can be found via a diamond sum of $K_{3,8}$ and $K_{4,8}$.
\end{proof}

\begin{figure}[tbp]
\centering
\includegraphics[scale=0.9]{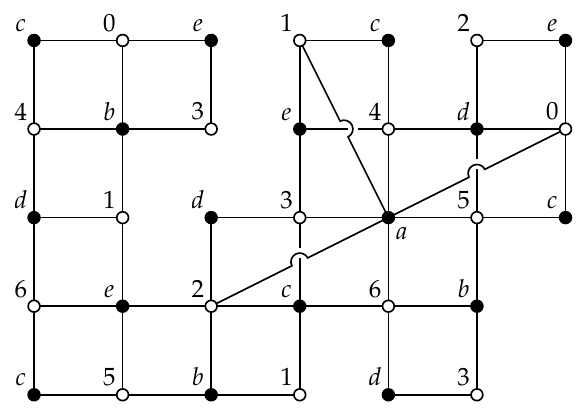}
\caption{A polyhedral net for $K_{5,7}-K_{1,3}$.}
\label{fig-k57}
\end{figure}

\begin{figure}[tbp]
\centering
\includegraphics[scale=0.9]{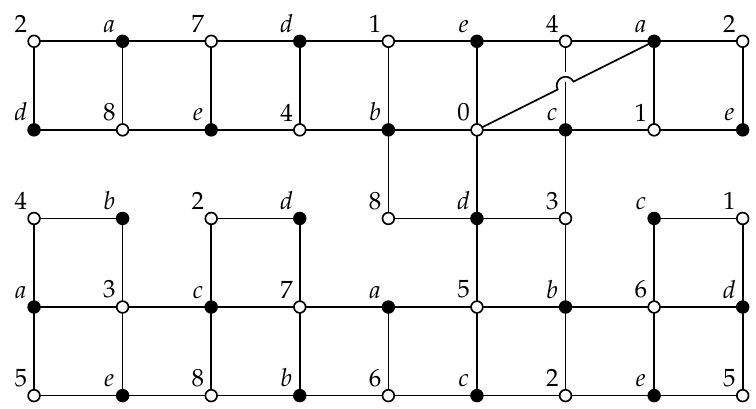}
\caption{A polyhedral net for $K_{5,9}-K_2$.}
\label{fig-k59}
\end{figure}

\begin{lemma}
There exist Kainen drawings of $K_{m,n}$ for all $m,n \geq 6$. 
\end{lemma}
\begin{proof}
By Corollary \ref{cor-induct}, it suffices to find Kainen drawings for $m,n \in \{6,7,8,9\}$. When $m = 6$, there exists a quadrangular embedding. For $m = 7$, a Kainen drawing of $K_{7,n}$ can be found by taking a diamond sum of $K_{6,n}$ and $K_{3,n}$. For $m = 9$, $n = 8,9$, we use a diamond sum of $K_{6,n}$ and $K_{5,n}$. The remaining case $K_{8,8}$ has a quadrilateral embedding. 
\end{proof}

\section{Nonorientable analogues}

The \emph{nonorientable genus} of $G$, denoted $\tilde{\gamma}(G)$, is defined to be the smallest value $k$ where $G$ has an embedding in the surface $N_k$, the sphere with $k$ crosscaps. For convenience, we allow $\tilde{\gamma}(G) = 0$, even though the sphere is orientable. The nonorientable genus of the complete and complete bipartite graphs is

\begin{equation}
\tilde{\gamma}(K_n) = H'(n) := \left\lceil \frac{(n-3)(n-4)}{6}\right\rceil, \text{~for~} n \geq 3, n \neq 7 \text{~and~}
\label{eq-genus-non-kn}
\end{equation}
\begin{equation}
\tilde{\gamma}(K_{m,n}) = H'(m,n) := \left\lceil \frac{(m-2)(n-2)}{2}\right\rceil, \text{~for~} m,n \geq 2.
\label{eq-genus-non-kmn}
\end{equation}

The analogues of Kainen's original conjectures would be to ask about the surface crossing numbers of these graphs in the surfaces of nonorientable genus
$$h'(n) := \left\lfloor \frac{(n-3)(n-4)}{6}\right\rfloor, \text{~and~} h'(m,n) := \left\lfloor \frac{(m-2)(n-2)}{2}\right\rfloor.$$
We use the notation $\tilde{\CR}_{k}(G)$ to denote the surface crossing number of $G$ in the surface $N_{k}$.

\subsection{Complete graphs}\label{sec-kn-non}

\begin{theorem}
The surface crossing number of the complete graph $K_n$, $n \geq 3$, in the nonorientable surface of genus $h'(n)$ is $$\tilde{\CR}_{h'(n)}(K_n) = \frac{(n-3)(n-4)}{2} \bmod{3},$$ except when $n = 7,8$. The crossing number of $K_7$ in $N_2$ is $1$, and the crossing number of $K_8$ in $N_3$ is $2$. 
\label{thm-non-kn}
\end{theorem}

The minimum triangulations problem for nonorientable surfaces has two counterexamples: Franklin \cite{Franklin-SixColor} showed that $K_7$ does not embed in the Klein bottle, and Ringel \cite{Ringel-K8} showed that $K_8-K_2$ does not embed in $N_3$. The two drawings of $K_7$ and $K_8$ shown in Figure \ref{fig-nonk7k8}, with 1 and 2 crossings, respectively, are subsequently optimal. When $n \equiv 0, 1, 3, 4, 6, 7, 9, 10 \Mod{12}$, and $n \geq 9$, there is a nonorientable triangular embedding of $K_n$. Thus, to prove Theorem \ref{thm-non-kn}, it suffices to show that when $n \equiv 2, 5, 8, 11 \Mod{12}$, $n \neq 8$, there is a triangular embedding of $K_n-K_2$ in a nonorientable surface that is a Kainen subembedding.  

\begin{figure}[tbp]
\centering
\includegraphics[scale=0.9]{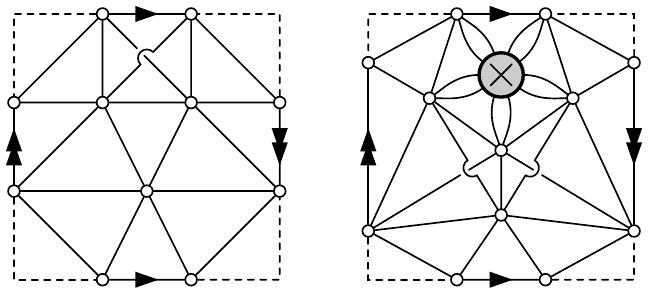}
\caption{$K_7$ and $K_8$ drawn in $N_2$ and $N_3$, respectively.}
\label{fig-nonk7k8}
\end{figure}

In the proof of the nonorientable Map Color Theorem presented in Ringel \cite{Ringel-MapColor}, some of the embeddings were constructed recursively by gluing together smaller embeddings, similar to the diamond sum operation we previously used in Section \ref{sec-kmn}. We borrow one such construction:

\begin{lemma}[Ringel \cite{Ringel-MapColor}, Theorem 10.6]
If there exists a nonorientable triangular embedding of $K_{2t+2}-K_2$, then there exists a nonorientable triangular embedding of $K_{4t+2}-K_2$. 
\end{lemma}
\begin{proof}
Ringel \cite{Ringel-MapColor} showed that there exists a triangular embedding $\phi$ of the graph $(K_{2t+1}\cup K_1)+C_{2t+1}$, $t \geq 2$, where $G+H$ is the graph join of $G$ and $H$, and $C_n$ is the cycle graph on $n$ vertices. Let $v$ be the vertex corresponding to $K_1$. Given a nonorientable triangular embedding $\phi'$ of $K_{2t+2}-K_2$, let $v'$ be any vertex that is adjacent to every vertex. By deleting $v$ and $v'$ and their incident faces from the embedded surfaces and identifying the resulting boundaries, we obtain a nonorientable triangular embedding of $K_{4t+2}-K_2$. 
\end{proof}

The proof can easily be extended to Kainen drawings by making sure $v'$ is a vertex with no responsibility:

\begin{corollary}
If there exists a nonorientable Kainen drawing of $K_{2t+2}$ with one crossing, then there exists one of $K_{4t+2}$ with one crossing.
\label{cor-ringelinduct}
\end{corollary}

\begin{lemma}
When $n \equiv 2,8 \Mod{12}$, $n > 8$, the surface crossing number of $K_n$ in the nonorientable surface of genus $h'(n)$ is 1.
\label{lem-case-non28}
\end{lemma}
\begin{proof}
The nonorientable triangular embeddings of $K_{12s+8}-K_2$, $s \geq 1$ constructed in both Korzhik \cite{Korzhik-Case8} and Sun \cite{Sun-Index2} are Kainen subembeddings of $K_{12s+8}$. A nonorientable Kainen subembedding of $K_{14}$ is constructed in Appendix \ref{app-table}. For all larger $n = 12s+2$, $s \geq 2$, we apply Corollary \ref{cor-ringelinduct} and induction with $t = 3s$.
\end{proof}

Recall that for a current graph embedded in a nonorientable surface, then $\alpha(e^+) = \alpha(e^-)$ for all twisted edges. In our drawings, such edges are depicted as ``broken'' and having two arrows pointing in opposite directions. For more details on how to determine the resulting derived embedding, see Section 4.4.5 of Gross and Tucker \cite{GrossTucker}. 

There are index 1 current graphs (see Section 8.3 of Ringel \cite{Ringel-MapColor}) that generate nonorientable triangular embeddings of $K_n-K_2$, for $n \equiv 5, 11 \Mod{12}$. We lift to index 3 so that the two vortices can be adjacent:

\begin{lemma}
When $n \equiv 5,11 \Mod{12}$, $n > 8$, the surface crossing number of $K_n$ in the nonorientable surface of genus $h'(n)$ is 1.
\label{lem-case-non511}
\end{lemma}
\begin{proof}
The planar crossing number of $K_5$ is 1, and a Kainen subembedding of $K_{11}$ is constructed in Appendix \ref{app-table}. For $n \geq 17$, the current graphs in Figures \ref{fig-non17}, \ref{fig-noncase5}, and \ref{fig-noncase11} generate Kainen subembeddings of $K_n$.
\end{proof}

\begin{figure}[tbp]
\centering
\includegraphics[scale=0.9]{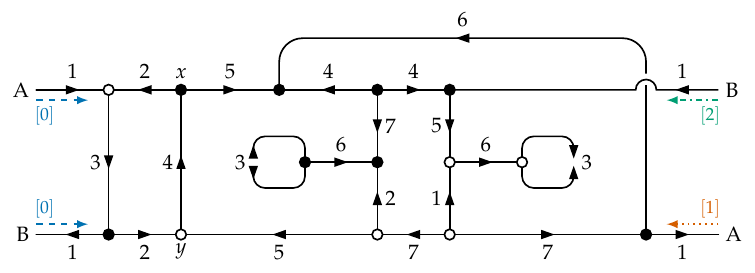}
\caption{A nonorientable index 3 current graph with group $\Z_{15}$.}
\label{fig-non17}
\end{figure}

\begin{figure}[tbp]
\centering
\includegraphics[scale=0.9]{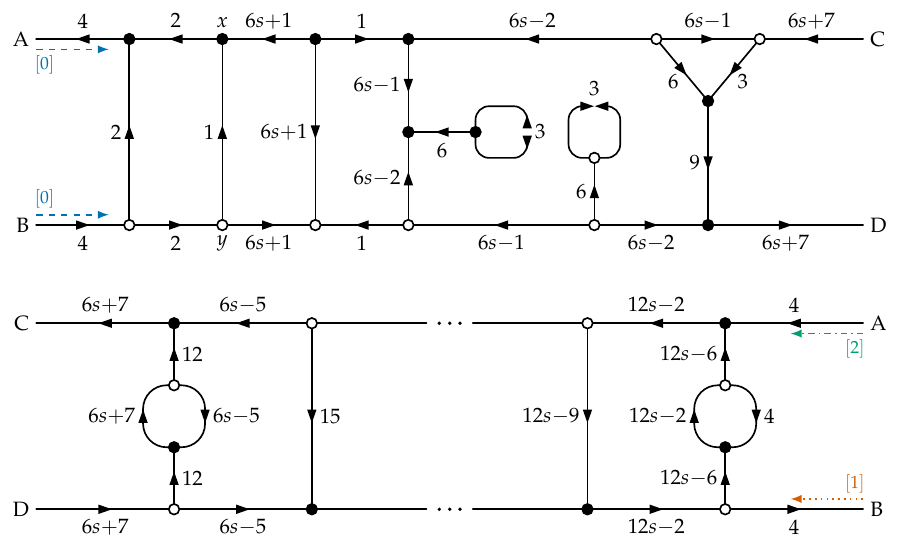}
\caption{A family of nonorientable index 3 current graphs with group $\Z_{12s+3}$, $s \geq 2$.}
\label{fig-noncase5}
\end{figure}

\begin{figure}[tbp]
\centering
\includegraphics[scale=0.9]{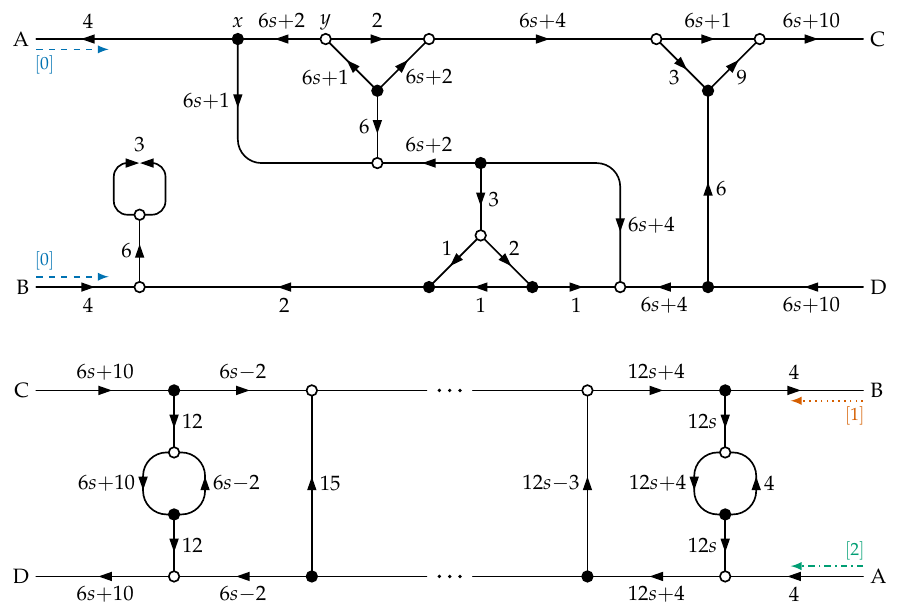}
\caption{A family of nonorientable index 3 current graphs with group $\Z_{12s+9}$, $s \geq 1$.}
\label{fig-noncase11}
\end{figure}

\subsection{Complete bipartite graphs} \label{sec-kmn-non}

The complete bipartite graph $K_{m,n}$, $m,n \geq 3$, has a quadrangular embedding in a nonorientable surface if and only if at least one of $m$ or $n$ is even. Thus, the nonorientable analogue of Kainen's conjecture for complete bipartite graphs treats the case when both $m$ and $n$ are odd.

\begin{theorem}
When $m$ and $n$ are odd and at least 3, the surface crossing number of $K_{m,n}$ in the surface of nonorientable genus $h'(m,n)$ is 1.
\label{thm-nonkmn}
\end{theorem}
\begin{proof}
When $m = 3$, double one right vertex of a quadrangular embedding of $K_{3,n-1}$ like in Lemma \ref{lem-k3n}. If $m > 3$, then take a diamond sum of this drawing of $K_{3,n}$ with a quadrangular embedding of $K_{m-1,n}$. 
\end{proof}

\section{Conclusion}

We solved Kainen's conjectures and their nonorientable analogues on surface crossing numbers of the complete graphs and complete bipartite graphs. However, the Kainen drawings with zero crossings came for free, since they are all previously known genus embeddings. Can one strengthen the results in this paper by replacing the functions $h(n)$ or $h(m,n)$ with $H(n)-1$ and $H(m,n)-1$, respectively? Solving this would determine, for each nonplanar $K_n$ or $K_{m,n}$, a nonzero crossing number in some surface. In Appendix \ref{app-minusone}, we show this for the complete bipartite graphs using the same techniques, but doing the same for complete graphs when $n \equiv 0, 3, 4, 7 \pmod{12}$ would be unwieldy using current graphs. 

For complete bipartite graphs, it is straightforward to extend our results to other surfaces by doubling more vertices or taking diamond sums where both drawings have crossings. What additional results can be obtained this way? 

We previously explained the relationship between Kainen drawings of complete graphs and minimal triangulations. An exact statement of Jungerman and Ringel's result on minimum triangulations is the following:
\begin{theorem}[Jungerman and Ringel \cite{JungermanRingel-Minimal}]
For each nonnegative integer $g \neq 2$, there exists a triangular embedding of a simple graph on 
$$M(g) := \left\lceil \,\frac{7+\sqrt{1+48g}}{2}\,\right\rceil$$
vertices in $S_g$, and this number is smallest possible. 
\end{theorem}
The surfaces of genus $h(n)$ constitute a vanishingly small fraction of all surfaces. We believe that Theorem \ref{thm-kn} and Jungerman and Ringel's result have a common generalization:
\begin{conjecture}
For each nonnegative integer $g \neq 2$, there exists a Kainen drawing of the complete graph $K_{M(g)}$ in $S_g$. 
\end{conjecture}
We have verified this conjecture up to $g = 17$ using computer search.

\bibliographystyle{alpha}
\bibliography{biblio}

\appendix

\section{Small rotation systems}\label{app-table}

In this section, we collect rotation systems of embeddings that are not derived from current graphs. All tables shown here describe orientable, triangular embeddings of simple graphs. 

Lee \cite{Lee-GenusCrossing} found a Kainen subembedding of $K_{11}$ by finding a special minimum genus embedding of $K_{10}$ and subdividing the nontriangular face. Since this work is unpublished, we converted her drawing into the rotation system shown in Table \ref{tab-k11}.
\begin{center}
$$\begin{array}{rlllllllllllllllllllllllllllll}
1. & 6 & 3 & 7 & 8 & 4 & 9 & 5 & 10 & 2 & 11 \\
2. & 1 & 10 & 6 & 7 & 4 & 8 & 5 & 9 & 3 & 11 \\
3. & 2 & 9 & 8 & 10 & 5 & 7 & 1 & 6 & 4 & 11 \\
4. & 3 & 6 & 10 & 9 & 1 & 8 & 2 & 7 & 5 & 11 \\
5. & 4 & 7 & 3 & 10 & 1 & 9 & 2 & 8 & 6 & 11 \\
6. & 5 & 8 & 9 & 7 & 2 & 10 & 4 & 3 & 1 & 11 \\
7. & 1 & 3 & 5 & 4 & 2 & 6 & 9 & 10 & 8 \\
8. & 1 & 7 & 10 & 3 & 9 & 6 & 5 & 2 & 4 \\
9. & 1 & 4 & 10 & 7 & 6 & 8 & 3 & 2 & 5 \\
10. & 1 & 5 & 3 & 8 & 7 & 9 & 4 & 6 & 2 \\
11. & 1 & 2 & 3 & 4 & 5 & 6
\end{array}$$
\captionof{table}{A Kainen subembedding of $K_{11}$.}\label{tab-k11}
\end{center}

The rotation system in Table \ref{tab-k20} is used to construct a Kainen drawing of $K_{20}$. It can be thought of as lifting a current graph of Jungerman and Ringel \cite{JungermanRingel-Minimal} to index 6. 

\begin{center}
$$\begin{array}{rrrrrrrrrrrrrrrrrrrrrrrrrrrrrr}
0. & 1 & 17 & x & 3 & 13 & 14 & 12 & 9 & 8 & 5 & 4 & 11 & 7 & 10 & 15 & 6 & 2 & w_0 & 16 \\
1. & 0 & 16 & x & 2 & 6 & 9 & 14 & 8 & 11 & 15 & 13 & 4 & 12 & 5 & 10 & 7 & 3 & w_1 & 17 \\
2. & 0 & 6 & 1 & x & 5 & 13 & 8 & 9 & 16 & 11 & 17 & 12 & 3 & 14 & 7 & 15 & 10 & 4 & w_0 \\
3. & 0 & x & 4 & 15 & 11 & 5 & w_1 & 1 & 7 & 17 & 9 & 10 & 14 & 2 & 12 & 6 & 16 & 8 & 13 \\
4. & 0 & 5 & 16 & 14 & 9 & 12 & 1 & 13 & 17 & 8 & 15 & 3 & x & 7 & 6 & w_0 & 2 & 10 & 11 \\
5. & 0 & 8 & 17 & 15 & 9 & 13 & 2 & x & 6 & 7 & w_1 & 3 & 11 & 14 & 10 & 1 & 12 & 16 & 4 \\
6. & 7 & 5 & x & 9 & 1 & 2 & 0 & 15 & 14 & 11 & 10 & 17 & 13 & 16 & 3 & 12 & 8 & w_0 & 4 \\
7. & 6 & 4 & x & 8 & 12 & 15 & 2 & 14 & 17 & 3 & 1 & 10 & 0 & 11 & 16 & 13 & 9 & w_1 & 5 \\
8. & 6 & 12 & 7 & x & 11 & 1 & 14 & 15 & 4 & 17 & 5 & 0 & 9 & 2 & 13 & 3 & 16 & 10 & w_0 \\
9. & 6 & x & 10 & 3 & 17 & 11 & w_1 & 7 & 13 & 5 & 15 & 16 & 2 & 8 & 0 & 12 & 4 & 14 & 1 \\
10. & 6 & 11 & 4 & 2 & 15 & 0 & 7 & 1 & 5 & 14 & 3 & 9 & x & 13 & 12 & w_0 & 8 & 16 & 17 \\
11. & 6 & 14 & 5 & 3 & 15 & 1 & 8 & x & 12 & 13 & w_1 & 9 & 17 & 2 & 16 & 7 & 0 & 4 & 10 \\
12. & 13 & 11 & x & 15 & 7 & 8 & 6 & 3 & 2 & 17 & 16 & 5 & 1 & 4 & 9 & 0 & 14 & w_0 & 10 \\
13. & 12 & 10 & x & 14 & 0 & 3 & 8 & 2 & 5 & 9 & 7 & 16 & 6 & 17 & 4 & 1 & 15 & w_1 & 11 \\
14. & 12 & 0 & 13 & x & 17 & 7 & 2 & 3 & 10 & 5 & 11 & 6 & 15 & 8 & 1 & 9 & 4 & 16 & w_0 \\
15. & 12 & x & 16 & 9 & 5 & 17 & w_1 & 13 & 1 & 11 & 3 & 4 & 8 & 14 & 6 & 0 & 10 & 2 & 7 \\
16. & 12 & 17 & 10 & 8 & 3 & 6 & 13 & 7 & 11 & 2 & 9 & 15 & x & 1 & 0 & w_0 & 14 & 4 & 5 \\
17. & 12 & 2 & 11 & 9 & 3 & 7 & 14 & x & 0 & 1 & w_1 & 15 & 5 & 8 & 4 & 13 & 6 & 10 & 16 \\
w_0. & 0 & 2 & 4 & 6 & 8 & 10 & 12 & 14 & 16 \\
w_1. & 1 & 3 & 5 & 7 & 9 & 11 & 13 & 15 & 17 \\
x. & 0 & 17 & 14 & 13 & 10 & 9 & 6 & 5 & 2 & 1 & 16 & 15 & 12 & 11 & 8 & 7 & 4 & 3
\end{array}$$
\captionof{table}{An orientable triangular embedding used to construct a Kainen drawing of $K_{20}$.}\label{tab-k20}
\end{center}

Two special nonorientable Kainen drawings are needed in Section \ref{sec-kn-non}. To circumvent the need for writing down and face-tracing general rotation systems, we present orientable embeddings of some spanning subgraphs in Tables \ref{tab-k11} and \ref{tab-k14}. The missing $P_3$ subgraphs can be incorporated into the embeddings by the crosscap augmentation shown in Figure \ref{fig-p3addition}. For both embeddings, we can set $v = 4$ and $(a,b,c,d) = (5, 6, 7, 8)$. For Table \ref{tab-k14}, the remaining $P_3$ can be recovered using $v = 9$ and $(a,b,c,d) = (10, 11, 12, 13)$. Both resulting embeddings are nonorientable and triangular, and in both embeddings, the remaining missing edge $(0,1)$ can be drawn across the edge $(2,3)$. 

\begin{center}
$$\begin{array}{rlllllllllllllllllllllllllllll}
0. & 2 & 4 & 6 & 9 & 5 & 8 & 10 & 7 & 3 \\
1. & 2 & 3 & 5 & 4 & 9 & 8 & 7 & 10 & 6 \\
2. & 0 & 3 & 1 & 6 & 5 & 10 & 8 & 9 & 7 & 4 \\
3. & 0 & 7 & 9 & 6 & 10 & 4 & 8 & 5 & 1 & 2 \\
4. & 0 & 2 & 7 & 8 & 3 & 10 & 9 & 1 & 5 & 6 \\
5. & 0 & 9 & 10 & 2 & 6 & 4 & 1 & 3 & 8 \\
6. & 0 & 4 & 5 & 2 & 1 & 10 & 3 & 9 \\
7. & 0 & 10 & 1 & 8 & 4 & 2 & 9 & 3 \\
8. & 0 & 5 & 3 & 4 & 7 & 1 & 9 & 2 & 10 \\
9. & 0 & 6 & 3 & 7 & 2 & 8 & 1 & 4 & 10 & 5 \\
10. & 0 & 8 & 2 & 5 & 9 & 4 & 3 & 6 & 1 & 7
\end{array}$$
\captionof{table}{An orientable triangular embedding of $K_{11}-(K_2 \cup P_3)$.}\label{tab-non-k11}
\end{center}

\begin{center}
$$\begin{array}{rlllllllllllllllllllllllllllll}
0. & 2 & 11 & 5 & 10 & 4 & 12 & 8 & 7 & 13 & 6 & 9 & 3\\
1. & 2 & 3 & 6 & 10 & 7 & 4 & 13 & 8 & 11 & 12 & 5 & 9\\
2. & 0 & 3 & 1 & 9 & 6 & 7 & 10 & 8 & 13 & 5 & 12 & 4 & 11\\
3. & 0 & 9 & 7 & 11 & 4 & 8 & 10 & 5 & 13 & 12 & 6 & 1 & 2\\
4. & 0 & 10 & 9 & 5 & 6 & 13 & 1 & 7 & 8 & 3 & 11 & 2 & 12\\
5. & 0 & 11 & 6 & 4 & 9 & 1 & 12 & 2 & 13 & 3 & 10\\
6. & 0 & 13 & 4 & 5 & 11 & 10 & 1 & 3 & 12 & 7 & 2 & 9\\
7. & 0 & 8 & 4 & 1 & 10 & 2 & 6 & 12 & 11 & 3 & 9 & 13\\
8. & 0 & 12 & 9 & 11 & 1 & 13 & 2 & 10 & 3 & 4 & 7\\
9. & 0 & 6 & 2 & 1 & 5 & 4 & 10 & 11 & 8 & 12 & 13 & 7 & 3\\
10. & 0 & 5 & 3 & 8 & 2 & 7 & 1 & 6 & 11 & 9 & 4\\
11. & 0 & 2 & 4 & 3 & 7 & 12 & 1 & 8 & 9 & 10 & 6 & 5\\
12. & 0 & 4 & 2 & 5 & 1 & 11 & 7 & 6 & 3 & 13 & 9 & 8\\
13. & 0 & 7 & 9 & 12 & 3 & 5 & 2 & 8 & 1 & 4 & 6
\end{array}$$
\captionof{table}{An orientable triangular embedding of $K_{14}-(K_2 \cup P_3 \cup P_3)$.}\label{tab-k14}
\end{center}

\begin{figure}[tbp]
\centering
\includegraphics[scale=0.9]{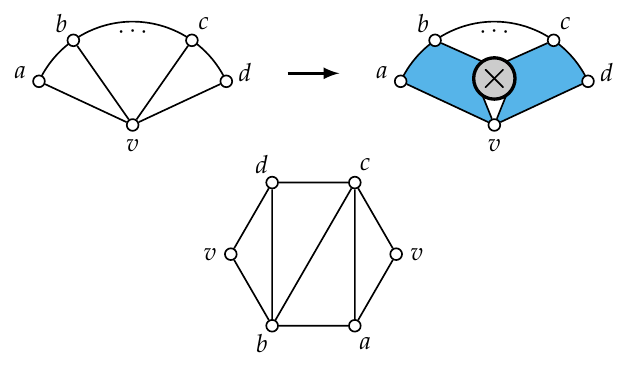}
\caption{Adding three edges with a crosscap.}
\label{fig-p3addition}
\end{figure}

\section{Kainen drawings of $K_{m,n}$ with four (or two) crossings}\label{app-minusone}

We calculate the crossing number of $K_{m,n}$ in the surface of genus $H(m,n)-1$. The complete bipartite graphs where $H(m,n)-1 < h(m,n)$ are the ones that have a quadrangular embedding in some orientable surface, i.e., when $m$ and $n$ are both even, or if one of $m$ or $n$ is congruent to 2 modulo 4. We show that Kainen's lower bound is tight except in one case.

\begin{theorem}
Let $m$ and $n$ be integers such that $m, n \geq 3$ and $(m-2)(n-2) \equiv 0 \Mod{4}$. Then $K_{m,n}$ has a Kainen drawing in $S_{H(m,n)-1}$ with four crossings, except $K_{3,6} = K_{6,3}$, which has planar crossing number 6. 
\label{thm-extendorient}
\end{theorem}
\begin{proof}
Kleitman \cite{Kleitman-K5n} proved that $K_{3,6}$ requires 6 crossings in the plane.

If $m, n \equiv 0 \Mod{4}$, then we can take the diamond sum of the drawings of $K_{m, 3}$ and $K_{m, n-1}$ with two crossings each constructed in the proof of Theorem \ref{thm-kmn}. Both drawings have right slackers by Lemma \ref{lem-k3n} and Proposition \ref{prop-pigeon}.

When $m = 6$ and $n \geq 4$, a few special constructions are needed. For $K_{6,4}$, we can double two left vertices in the toroidal embedding of $K_{4,4}$. A Kainen drawing of $K_{6,5}$ with four crossings and a right slacker (vertex $b$ or $c$) is shown in Figure \ref{fig-k56}. When $n \geq 6$, we can take the diamond sum of the aforementioned drawing of $K_{6,5}$ with a quadrangular embedding of $K_{6,n-3}$. 

Finally, suppose that $m \equiv 2 \Mod{4}$ and $m \geq 10$. By doubling four left vertices in a quadrangular embedding of $K_{m-4,3}$, we obtain a Kainen drawing of $K_{m, 3}$ with four crossings that has a right slacker. We take the diamond sum of this drawing with a quadrangular embedding of $K_{m, n-1}$. 

\end{proof}

\begin{figure}[tbp]
\centering
\includegraphics[scale=0.9]{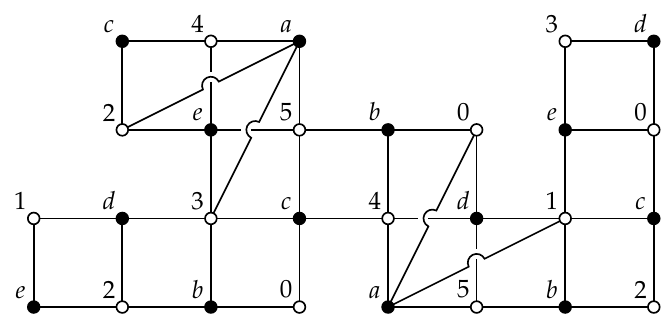}
\caption{A polyhedral net for $K_{5,6}-K_{1,4}$.}
\label{fig-k56}
\end{figure}

We conclude with the analogue for nonorientable embeddings:

\begin{theorem}
Let $m$ and $n$ be integers such that $3 \leq m \leq n$, where at least one of $m$ or $n$ is even. Then $K_{m,n}$ has a Kainen drawing in $N_{H'(m,n)-1}$ with two crossings. 
\label{thm-extendnonorient}
\end{theorem}
\begin{proof}
The claim is true for $m = 3$ (or $n = 3$) because one can double two vertices in any nonorientable quadrangular embedding of $K_{3, n-2}$ to obtain a drawing of $K_{3,n}$ with two crossings. If $m$ and $n$ are both even, we can form the diamond sum of the drawing of $K_{3,n}$ with two crossings we just constructed and a quadrangular embedding of $K_{m-1,n}$. If $m$ is even and $n$ is odd, we can take the diamond sum of the Kainen drawings of $K_{3, n}$ and $K_{m-1, n}$ with one crossing each, constructed in Theorem \ref{thm-nonkmn}.
\end{proof}

\end{document}